\documentclass{article} 

\usepackage{ifthen}

\usepackage{amsfonts,amsmath,amssymb,algorithm,amsthm,subfigure,multicol,multirow}
\usepackage[noend]{algpseudocode}
\usepackage{varwidth}
\makeatletter
\renewcommand{\ALG@beginalgorithmic}{\small}
\makeatother
\usepackage{nips15submit_e,times}
\usepackage{url}
\usepackage[pdftex]{graphicx}
\usepackage[section]{placeins}

\newcommand{\half}{{\textstyle\frac12}}

\newcommand{\Acal}{\mathcal A}
\newcommand{\Pcal}{\mathcal P}
\newcommand{\Ncal}{\mathcal N}
\newcommand{\Vcal}{\mathcal V}
\newtheorem{theorem}{Theorem}[section]

\newtheorem{example}{Example}[section]

\nipsfinalcopy 

\title{Primal-Dual Active-Set Methods for Isotonic Regression and Trend Filtering}

\author{
Zheng Han \\
Industrial and Systems Engineering\\
Lehigh University\\
Bethlehem, PA 18015 \\
\texttt{zhh210@lehigh.edu}
\And
Frank E.~Curtis \\
Industrial and Systems Engineering\\
Lehigh University\\
Bethlehem, PA 18015 \\
\texttt{frank.e.curtis@gmail.com} 
}

\begin{document}

\maketitle

\begin{abstract}
  Isotonic regression (IR) is a non-parametric calibration method used in supervised learning.  For performing large-scale IR, we propose a primal-dual active-set (PDAS) algorithm which, in contrast to the state-of-the-art Pool Adjacent Violators (PAV) algorithm, is easily warm-started and can be parallelized.  This warm-starting capability makes it well-suited for online settings.  We prove that, like the PAV algorithm, our PDAS algorithm for IR is convergent and has a work complexity of $\mathcal{O}(n)$, though our numerical experiments suggest that our PDAS algorithm is often faster than PAV.  In addition, we propose PDAS variants (with safeguarding to ensure convergence) for solving related trend filtering (TF) problems, providing the results of experiments to illustrate their effectiveness.
\end{abstract}

\section{Introduction}

Isotonic regression is a non-parametric method for fitting an arbitrary monotone function to a dataset \cite{barlow1972isotonic,brunk1972statistical} that has recently gained favor as a calibration method for supervised learning \cite{Graepel10,McMahan13,Niculescu-Mizil05,Zadrozny02}.  A well-known and efficient method for solving IR problems is the Pool Adjacent Violators (PAV) algorithm \cite{best90active}.  This method is easily implemented and enjoys a convergence guarantee with a work complexity of $\mathcal O(n)$ where $n$ is the dimension of the dataset.  A drawback of the PAV algorithm in large-scale settings, however, is that it is inherently sequential.  Consequently, in order to exploit parallelism, one has to resort to decomposing the IR problem \cite{kearsley1996approach}, where deciding the number of processors is nontrivial.  For example, a recent Spark implementation of a parallelized PAV method suffers from significant overhead \cite{isospark}. In addition, since the PAV algorithm must be initialized from a particular starting point, it cannot be warm-started---a fact that is especially detrimental when a sequence of IR problems need to be solved \cite{suo2014ordered}, such as in online settings where data points are added continually.  As an alternative, we propose a primal-dual active-set (PDAS) method for solving IR problems.  Our PDAS algorithm also has a convergence guarantee for IR and a work complexity of $\mathcal O(n)$, but can be warm-started and is easily parallelized.  We also provide PDAS algorithm variants for a related class of trend filtering (TF) problems.  Alternative approaches for certain TF problems include interior point methods \cite{kim09}, specialized ADMM methods \cite{ramdas2014fast}, and proximal methods \cite{barbero2014modular}, but advantages of active-set methods such as ours are that they may terminate finitely and often yield very accurate solutions.  For reasons such as these, they may be favorable for certain applications such as switch point identification and time series segmentation.  The results of numerical experiments are provided for all of our algorithm variants to illustrate their practical strengths.

This paper is organized as follows.  In \S\ref{sec:prob.description}, we define and describe the IR and TF problems for which our algorithms are designed. In \S\ref{sec:alg.isoreg}, we summarize and compare the well-known Pool Adjacent Violators (PAV) algorithm and our proposed primal-dual active-set (PDAS) method for solving IR problems.  PDAS variants (with and without safeguards) for solving related TF problems are presented in \S\ref{sec:pdas.tf}.  We report on our experimental results as well as our findings with other approaches in~\S\ref{sec:experiments}.  Finally, concluding remarks are provided in \S\ref{sec:conclusion}.

\section{Problem Descriptions}\label{sec:prob.description}

\paragraph{Isotonic Regression}  We consider the isotonic regression (IR) problem
\begin{equation}\label{prob.isoreg}\tag{IR}
  \min_{{\theta}\in\mathbb R^n} \half \sum_{i=1}^n \omega_i(y_i -\theta_i)^2\ \ \text{subject to}\ \ \theta_1 \leq \ldots \leq \theta_n,
\end{equation}
where $y \in \mathbb R^n$ represents observed data and $\omega \in \mathbb R^n_+$ represents weights for the data fitting term.  The goal of this optimization problem formulation is to determine a monotonically increasing step function that matches the observed data as closely as possible in a sense of distance defined by $\omega$.

\paragraph{Trend Filtering} Problem~\eqref{prob.isoreg} is related to a special case of the trend filtering problem
\begin{equation}\label{prob.general}\tag{TF}
  \min_{\theta\in\mathbb R^n}\ \phi(\theta),\ \ \text{where}\ \ \phi(\theta) = f(\theta) + \lambda g(\theta),
\end{equation}
$f : \mathbb R^n \mapsto \mathbb R$ is smooth and convex, $\lambda > 0$ is a regularization parameter, and the regularization function $g: \mathbb R^n \mapsto \mathbb R$ is convex but not necessarily smooth. In trend filtering or time series segmentation, $f$ is usually chosen to measure the distance between $\theta$ and $y$ while $g$ imposes desired properties on the solution $\theta$. A typical trend filtering problem has the form
\begin{equation*}
  f(\theta) =  \half \sum_{i=1}^n \omega_i(y_i -\theta_i)^2\ \ \text{with}\ \ g(\theta) = \|D\theta\|_1\ \ \text{or}\ \ g(\theta) = \|(D\theta)_+\|_1,
\end{equation*}
where $D$ is a first-order (or higher) difference operator and $(\gamma)_+ = \max\{\gamma,0\}$ (component-wise).
Specifically, as described in \cite{kim09}, a $k$-th order difference matrix $D^{(k,n)}\in\mathbb R^{(n-k)\times n}$ is defined recursively via the relation $D^{(k,n)} = D^{(1,n-k+1)}D^{(k-1,n)}$.  For example, the first and second order difference matrix $D^{(1,n)}$ and $D^{(2,n)}$ are defined, respectively, as
\begin{equation*}\small
  D^{(1,n)} =
  \begin{bmatrix}
    1 & -1 &    &   &  \\
      & 1  & -1 &   &  \\
      &    & \ddots & \ddots & \\
      &    &    & 1 & -1 \\
  \end{bmatrix}\mbox{ and }
  D^{(2,n)} =
  \begin{bmatrix}
    1 & -2 & 1   &   & & \\
      & 1  & -2 & 1  &  & \\
      &    & \ddots & \ddots & \ddots & \\
      &    &    & 1 & -2 &  1\\
  \end{bmatrix}.
\end{equation*}
When $g(\theta) = \|(D^{(1,n)}\theta)_+\|_1$ and $\lambda$ is sufficiently large, solving \eqref{prob.general} gives the solution of \eqref{prob.isoreg}.

\section{Algorithms for Isotonic Regression}\label{sec:alg.isoreg}

We describe and compare two efficient algorithms for solving problem~\eqref{prob.isoreg}; in particular, the PAV and our proposed PDAS algorithms are described and their corresponding theoretical properties are summarized in \S\ref{subsec:isoreg.pav} and \S\ref{subsec:isoreg.pdas}, respectively. Throughout this and the subsequent sections, we borrow the following notation from \cite{best90active} in the algorithm descriptions.

\paragraph{Notation} Let $J$ represent a partition of the variable indices $\{1,2,\ldots,n\}$ into ordered blocks $\{B_1,B_2,\ldots\}$ where each block consists of consecutive indices, i.e., each block has the form $\{p,\ldots,q\}$ for $p\leq q$. The immediate predecessor (successor) of block $B$ is denoted as $B_{-}$ ($B_{+}$).  By convention, $B_{-}$ ($B_{+}$) equals $\emptyset$ when $B$ is the initial (final) block.  The weighted average of the elements of $y$ in block $B$ is denoted $\mbox{Av}(B) := (\sum_{i=p}^q \omega_i y_i)/(\sum_{i=p}^q \omega_i)$. For each index $i \in B = \{p,\dots,q\}$, we define the ``lower'' and ``upper'' sets
\begin{equation*}
  L_i(J) = \{p,p+1,\ldots,i\}\ \ \text{and}\ \ U_i(J) = \{i+1,i+2,\ldots, q\}.
\end{equation*}
Hereinafter, we shall use $L_i$ and $U_i$ for brevity when their dependence on a particular $J$ is clear.

\subsection{PAV for Isotonic Regression}\label{subsec:isoreg.pav}

We describe briefly the PAV algorithm \cite{best90active,barlow1972isotonic,tibshirani2011nearly,grotzinger1984projections,brunk1972statistical,kearsley1996approach}, state its main theoretical properties, and discuss its important features in this section.  To start with, we rephrase the PAV algorithm from \cite{best90active}, where the algorithm is shown to replicate a dual active-set method for quadratic optimization.

\begin{algorithm}[H]\small
  \caption{PAV for Isotonic Regression}
  \label{alg.pava}
  \begin{algorithmic}[1]
    \State Input the initial partition $J = \{\{1\},\ldots,\{n\}\}$ and set $C=\{1\}$
    \Loop
      \If{$C_+ = \emptyset$}
        \State Terminate and \textbf{return} $\theta_i = \mbox{Av}(B)$ for each $i\in B$ for each $B\in J$
      \EndIf
      \If{$\mbox{Av}(C)\leq \mbox{Av}(C_+)$}
        \State Set $C\gets C_+$
        \Else
        \State Set $J \gets \big( J\backslash \{C,C_+\}\big)\cup (C\cup C_+)$ and $C \gets C\cup C_+$
        \Comment{Merge $C$ with $C_+$}
        \While{$\mbox{Av}(C_-) > \mbox{Av}(C)$ and $C_-\neq\emptyset$}
          \State Set $J \gets \big( J\backslash \{C_-,C\}\big)\cup (C_-\cup C)$ and $C \gets C_-\cup C$
          \Comment{Merge $C$ with $C_-$}
        \EndWhile
      \EndIf
    \EndLoop
  \end{algorithmic}
\end{algorithm}

The main idea of Algorithm~\ref{alg.pava} can be understood as follows.  Initially, each index is represented by a separate block.  The algorithm then sequentially visits all blocks, merging a block with its successor whenever a ``violator'' is met, i.e., whenever a block has a weighted average greater than its successor.  Once any merge occurs, the algorithm searches backwards to perform subsequent merges in order to ensure that, at the end of any \textbf{loop} iteration, no violators exist up to the furthest visited block.  Once all blocks have been visited, no violator exists and the solution $\theta$ will be monotonically increasing.

An impressive property of Algorithm~\ref{alg.pava} is that by storing an intermediate value for each block and showing that at most $n$ merge operations may occur, one obtains an efficient implementation that solves problem~\eqref{prob.isoreg} within $\mathcal O(n)$ elementary arithmetic operations \cite{grotzinger1984projections}.  Due to this fact and its good practical performance, Algorithm~\ref{alg.pava} has been popular since its invention.  However, we argue that the PAV algorithm does have critical drawbacks when it comes to solving large-scale problems of interest today.  First, Algorithm~\ref{alg.pava} must be initialized with $J = \{\{1\},\dots,\{n\}\}$, which is unfortunate when one has a better initial partition, such as when one is solving a sequence of related instances of~\eqref{prob.isoreg} \cite{suo2014ordered}.  In addition, the sequential nature of PAV makes it difficult to leverage multi-processor infrastructures.

\subsection{PDAS for Isotonic Regression}\label{subsec:isoreg.pdas}

Primal-dual active-set (PDAS) methods have been proposed in the literature for solving Linear Complementarity Problems (LCPs) \cite{Agan84}, bound-constrained QPs (BQPs) \cite{HintItoKuni03}, and generally-constrained QPs \cite{CurtHanRobi14}. To our knowledge, however, the application and theoretical analysis of a tailored PDAS method for solving problem~\eqref{prob.isoreg} has not previously been studied.

In this section, we propose a PDAS method designed explicitly for \eqref{prob.isoreg} and discuss its theoretical guarantees and practical benefits. We first reveal the relationship between problem~\eqref{prob.isoreg} and a special class of convex BQPs for which a primal-dual active-set (PDAS) method is known to be well-suited. We then propose our PDAS algorithm tailored for solving \eqref{prob.isoreg}.  Finally, the complexity of our proposed algorithm is analyzed and its key features and properties are discussed.

\paragraph{\eqref{prob.isoreg} and \eqref{prob.dual}}  Let $\Omega=\mbox{diag}(\omega_1,\ldots,\omega_n)$ be the diagonal weight matrix.  The dual of \eqref{prob.isoreg} is then
\begin{equation}\label{prob.dual}\tag{BQP}
  \min_{z\in\mathbb R_+^{n-1}} \half z^TD\Omega^{-1}D^Tz - y^TD^Tz,
\end{equation}
where
\begin{equation*}\small
  D\Omega^{-1}D^T = 
  \begin{bmatrix}
    \frac{2}{\omega_1} & -\frac{1}{\omega_2} & 0 & \cdots & 0\\
    -\frac{1}{\omega_2} & \frac{2}{\omega_2} &-\frac{1}{\omega_3} & \cdots & 0\\
    0 & -\frac{1}{\omega_3} & \frac{2}{\omega_3} & \ddots & 0\\
    \vdots&\vdots&\ddots&\ddots&-\frac{1}{\omega_{n-2}}\\
    0 & 0 & \cdots & -\frac{1}{\omega_{n-2}} & \frac{2}{\omega_{n-1}}\\
  \end{bmatrix}
  \ \ \text{and}\ \ 
  Dy = 
  \begin{bmatrix}
    y_1 - y_2\\
    y_2 - y_3\\
    \vdots\\
    y_{n-1}-y_n
  \end{bmatrix}.
\end{equation*}
Since $D\Omega^{-1}D^T$ is positive definite with non-positive off-diagonal entries, it is an $M$-matrix, meaning that \eqref{prob.dual} has a form for which a PDAS method is well-suited \cite{HintItoKuni03}.  In descriptions of PDAS methods for BQP such as that in \cite{HintItoKuni03}, the notion of a partition corresponds to a division of the index set for $z$ into ``active'' and ``inactive'' sets.  There is a one-to-one correspondence between a partition of \eqref{prob.dual} and that of \eqref{prob.isoreg}; specifically, the non-zero indices in $z$ correspond to the boundaries of the blocks of a partition for problem~\eqref{prob.isoreg}.  We have the following result for the present setting.

\begin{theorem}[Theorem 3.2, \cite{HintItoKuni03}]\label{corol.property}
  If the PDAS method from \cite{HintItoKuni03} is applied to solve problem~\eqref{prob.dual}, then the iterate sequence $\{z_k\}$ is nondecreasing, has $z_k \geq 0$ for all $k\geq 1$, and converges to the optimal solution of \eqref{prob.dual}.
\end{theorem}

\paragraph{A PDAS Algorithm for \eqref{prob.isoreg}}  One can apply the PDAS method from \cite{HintItoKuni03} to solve~\eqref{prob.isoreg} by applying the approach to its dual~\eqref{prob.dual}.  However, a straightforward application would fail to exploit the special structure of problem~\eqref{prob.isoreg}.  Algorithm~\ref{alg.pdas.practical}, on the other hand, generates the same sequence of iterates as the PDAS method of \cite{HintItoKuni03}, but is written in a much more computationally efficient form.

\begin{algorithm}[ht]\small
  \caption{PDAS for Isotonic Regression}
  \label{alg.pdas.practical}
  \begin{algorithmic}[1]
    \State Input an initial partition $J_0$ \label{step.pdas.practical.init.start}
        \State For each $i \in B = \{p,\ldots,q\} \in J_0$, set 
        \begin{equation*}
            \theta_i \gets \frac{\sum_{i\in B} \omega_i y_i}{\sum_{i\in B} \omega_i}\ \ \text{and}\ \
            z_i \gets
            \begin{cases}
              \omega_i(y_i - \theta_i) & \mbox{ if }i= p \\
              0 & \mbox{ if }i=q\neq n \\
              z_{i-1} +\omega_i(y_i - \theta_i) & \mbox{ otherwise }
            \end{cases}
        \end{equation*}
      \State Initialize $J_1 \gets J_0$
      \State \textbf{for each} $i\in B \in J_1$ with $z_i < 0$, set $J_1 \gets (J_1\backslash B) \cup \{L_i, U_i\}$\label{step.pdas.practical.split}
      \Comment{Split $B$ into $L_i$ and $U_i$}
      \State \textbf{for each} $B\in J_1$, set $\alpha_B \gets \sum_{i\in B} \omega_i y_i$, $\beta_B \gets \sum_{i\in B} \omega_i$, and $\mu_B \gets \alpha_B/\beta_B$\label{step.pdas.practical.init.end}
    \For {$k = 1,2,\dots$} \label{step.pdas.practical.begin.loop}
      \State Initialize $J_{k+1}\gets J_k$
      \State \textbf{for each} $\{B_s,\dots,B_t\} \subseteq J_k$ with $\mu_{B_{s-1}}\leq \mu_{B_s}>\cdots>\mu_{B_t}\leq \mu_{B_{t+1}}$ \Comment{Merge $\{B_s,\dots,B_t\}$}
        \begin{align*}\small
          \mbox{Let }& N\gets \bigcup_{j=s}^t B_j \ \mbox{and update}\ \ J_{k+1}\gets (J_{k+1}\cup N)\backslash\{B_s,\ldots,B_t\},\\
          \mbox{Set }&  \alpha_N \gets \sum_{j=s}^t \alpha_{B_j},\ \ \beta_N \gets \sum_{j=s}^t \beta_{B_j},\ \mbox{and}\ \mu_N \gets \alpha_N/\beta_N
        \end{align*}
      \State \textbf{if }$J_{k+1} = J_k$, \textbf{then} break\label{step.pdas.practical.end.loop}
    \EndFor
  \State \textbf{for each} $B\in J_k$, set $\theta_i\gets \mu_B$ for all $i\in B$
  \end{algorithmic}
\end{algorithm}

The major difference between Algorithms~\ref{alg.pava} and \ref{alg.pdas.practical} is the manner in which blocks are merged.  In Algorithm~\ref{alg.pava}, each merge involves a single pair of adjacent blocks, whereas Algorithm~\ref{alg.pdas.practical} allows simultaneous merge operations, each of which may involve more than two blocks.  In the following theorem, we prove that Algorithm~\ref{alg.pdas.practical}, like the careful implementation \cite{grotzinger1984projections} of Algorithm~\ref{alg.pava}, solves problem~\eqref{prob.isoreg} in $\mathcal O(n)$ elementary arithmetic operations.

\begin{theorem}\label{theorem.complexity.pdas}
  If Algorithm~\ref{alg.pdas.practical} is applied to solve problem~\eqref{prob.isoreg}, then it will yield the optimal solution for \eqref{prob.isoreg} within $\mathcal O(n)$ elementary arithmetic operations.
\end{theorem}
\begin{proof}
  One can verify that Algorithm~\ref{alg.pdas.practical} is equivalent to applying the PDAS method from \cite{HintItoKuni03} to solve \eqref{prob.dual}, from which it follows by Theorem~\ref{corol.property} that it finds the optimal solution in finitely many iterations. The initialization process in Steps~\ref{step.pdas.practical.init.start}--\ref{step.pdas.practical.init.end} requires $\mathcal O(n)$ elementary arithmetic operations as each step involves at most a constant number of calculations with each value from the dataset.  As for the main loop involving Steps~\ref{step.pdas.practical.begin.loop}--\ref{step.pdas.practical.end.loop}, the introduction of $\alpha$ and $\beta$ ensures that the number of elementary arithmetic operations in merging two blocks becomes $\mathcal O(1)$.  Thus, since the \textbf{for} loop only involves merge operations and there can be at most $n$ merges, the desired result follows.
\end{proof}

\subsection{Further Discussion}\label{subsec:isoreg.comparison}

Algorithm~\ref{alg.pdas.practical} enjoys several nice features that Algorithm~\ref{alg.pava} and other relevant algorithms do not possess.  For one thing, the initial partition $J_0$ of Algorithm~\ref{alg.pdas.practical} can be chosen arbitrarily.  This allows the algorithm to be warm-started if one has a good optimal partition estimate.  This feature is particularly appealing when a sequence of related instances of \eqref{prob.isoreg} need to be solved \cite{suo2014ordered}.

Another interesting feature of Algorithm~\ref{alg.pdas.practical} is its potential to be parallelized.  This is possible since Algorithm~\ref{alg.pdas.practical} allows for multiple independent merge operations in each iteration; see Step~8 of Algorithm~\ref{alg.pdas.practical}.  As an illustrative example with $y = \{6,4,2,9,11,4\}$ and $\omega = e$, we demonstrate in Figure~\ref{fig:pdas.merge} the different behavior between Algorithm~\ref{alg.pdas.practical} and Algorithm~\ref{alg.pava} when applied on this data set.

\begin{figure}[ht]
\centering
  \includegraphics[height=2.4in]{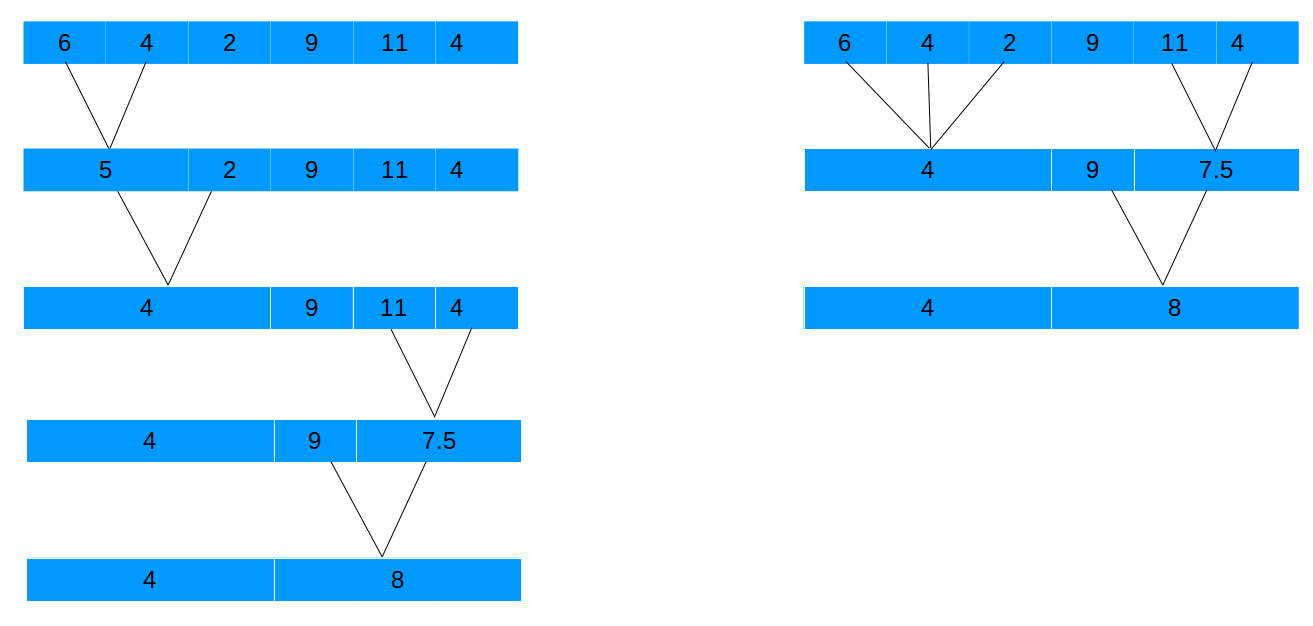}
\caption{Illustration of per-iteration merge operations in PAV (left) and PDAS (right).}
\label{fig:pdas.merge}
\end{figure}
As illustrated in Figure~\ref{fig:pdas.merge}, each iteration of PAV only merges two consecutive blocks whereas PDAS allows multiple consecutive blocks to be merged simultaneously throughout the dataset.  Notice also that PDAS only requires three division operations while PAV requires four, despite the fact that in both methods the number of merge operations are counted as four.

\section{PDAS for Trend Filtering}\label{sec:pdas.tf}

The trend filtering problem~\eqref{prob.general} can be viewed as a generalization of problem~\eqref{prob.isoreg}.  While \eqref{prob.isoreg} imposes monotonicity on the solution vector $\theta$, variants of \eqref{prob.general} can impose other related properties, as illustrated in \S\ref{subsec:tf.reg}. Consequently, it is natural to extend PDAS for solving \eqref{prob.general}, as we do in \S\ref{subsec:tf.pdas}. However, since a direct application of a PDAS method may cycle when solving certain versions of~\eqref{prob.general}, we propose safeguarding strategies to ensure convergence; see~\S\ref{subsec:tf.safeguard}.

\subsection{Regularization with Difference Operators}\label{subsec:tf.reg}

Common choices for the regularization function in problem~\eqref{prob.general} are $g(\theta) = g_1(\theta):= \|D^{(d,n)}\theta\|_1$ or $g(\theta) = g_{1+}(\theta) = \|(D^{(d,n)}\theta)_+\|_1$, where $D^{(d,n)}\in\mathbb R^{(n-d)\times n}$ is the $d$-order difference matrix on~$\mathbb R^{n}$. The choice of the regularization function determines the properties that one imposes on $\theta$.  We illustrate the typical behavior of $\theta$ for different choices of the regularization in Figure~\ref{fig:piecewise}.
\begin{figure}[htb]
\centering
\subfigure[$g = g_{1+}$, $d=1$]{
  \includegraphics[width=2in]{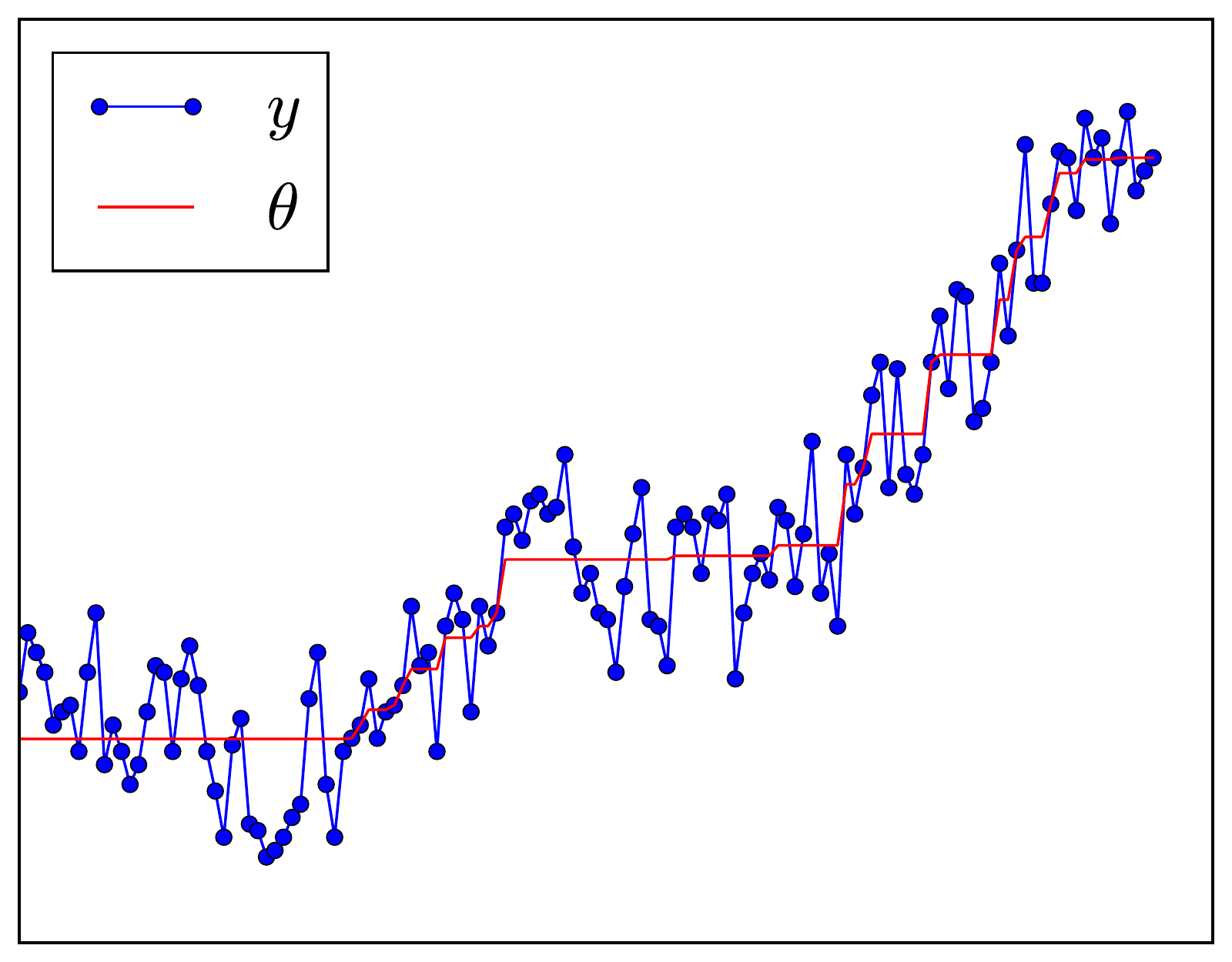}
 }
\subfigure[$g = g_1$, $d=1$]{
  \includegraphics[width = 2in]{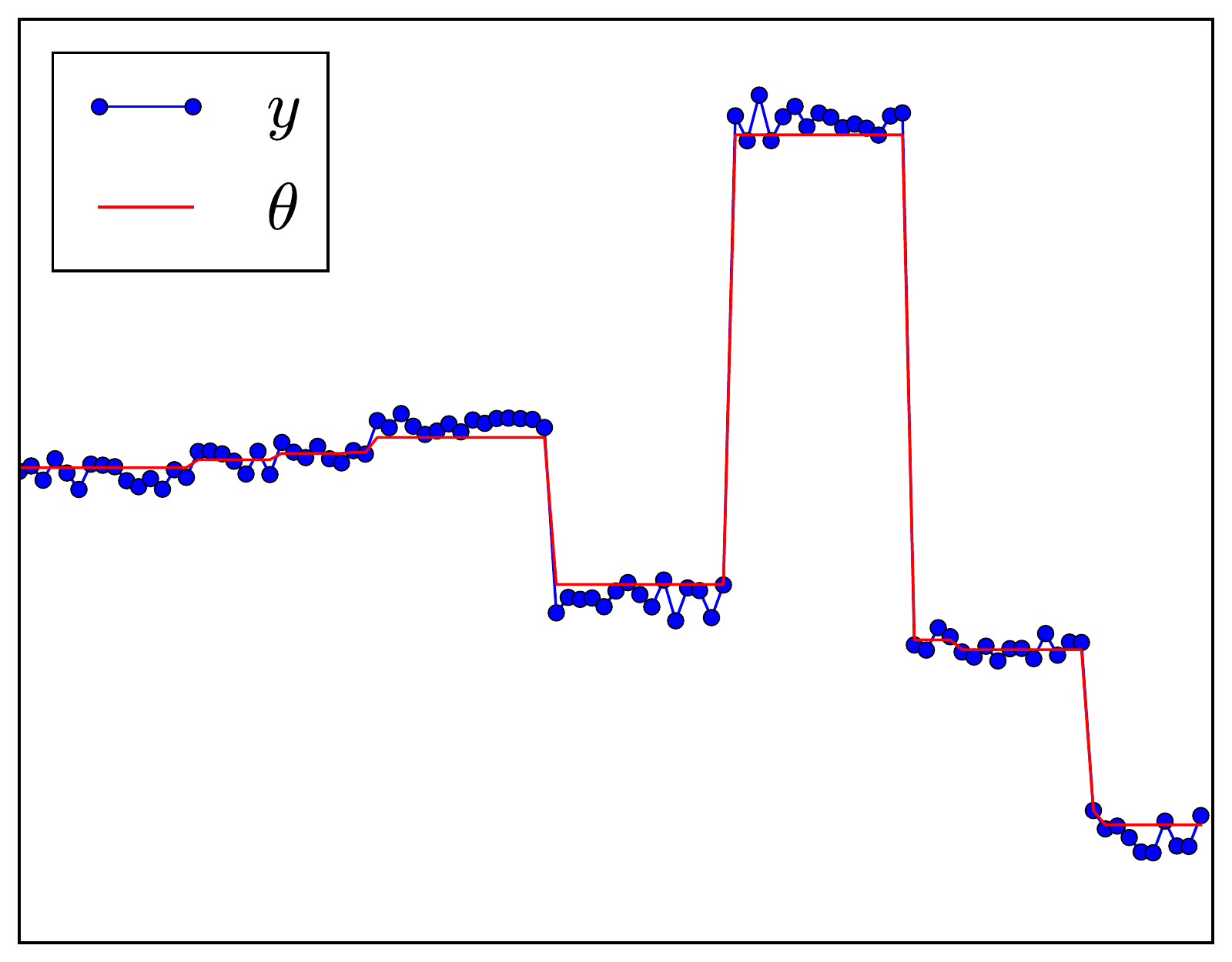}

}
\subfigure[$g=g_{1+}$, $d=2$]{
  \includegraphics[width = 2in]{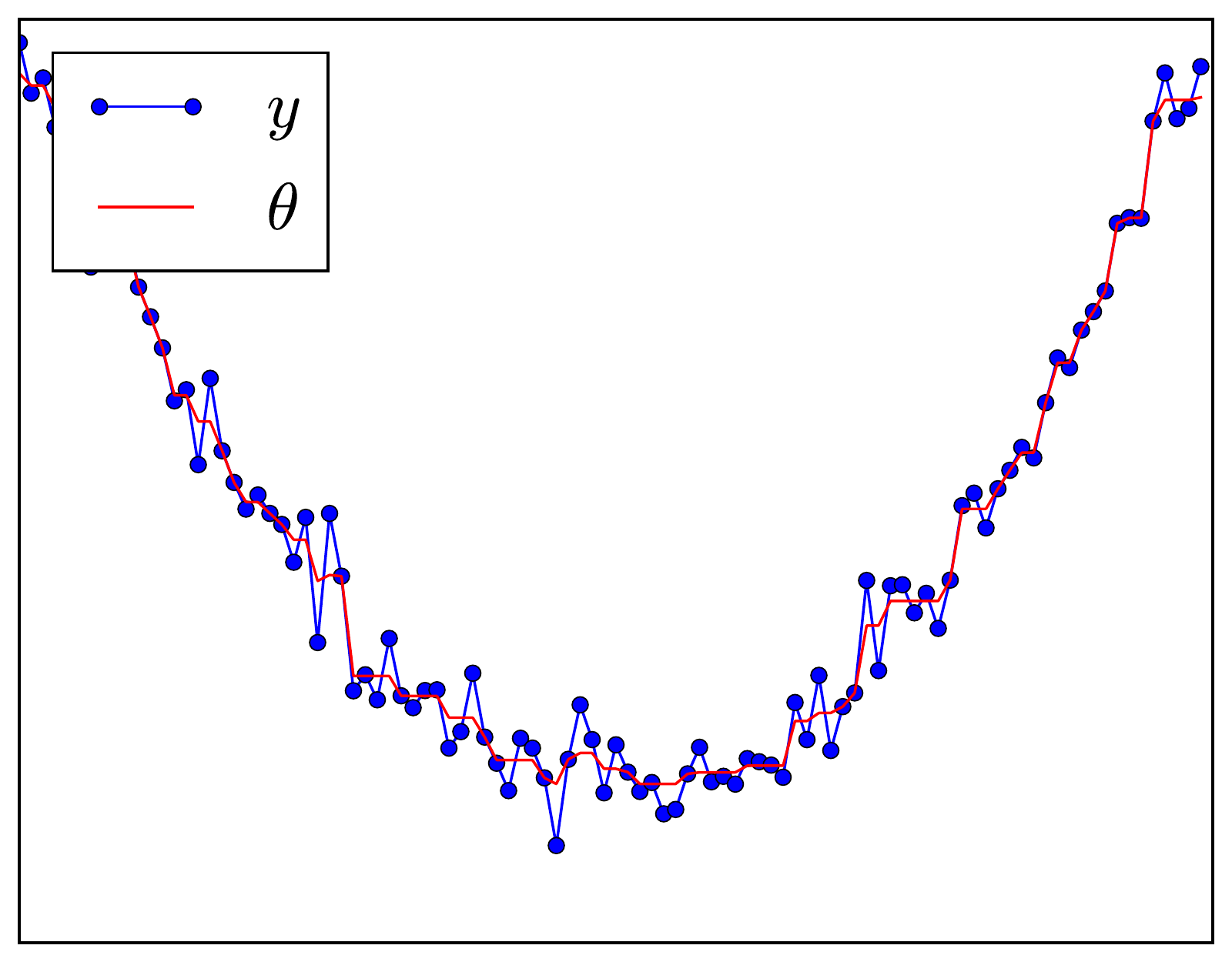}
}
\subfigure[$g=g_1$, $d=2$]{
  \includegraphics[width = 2in]{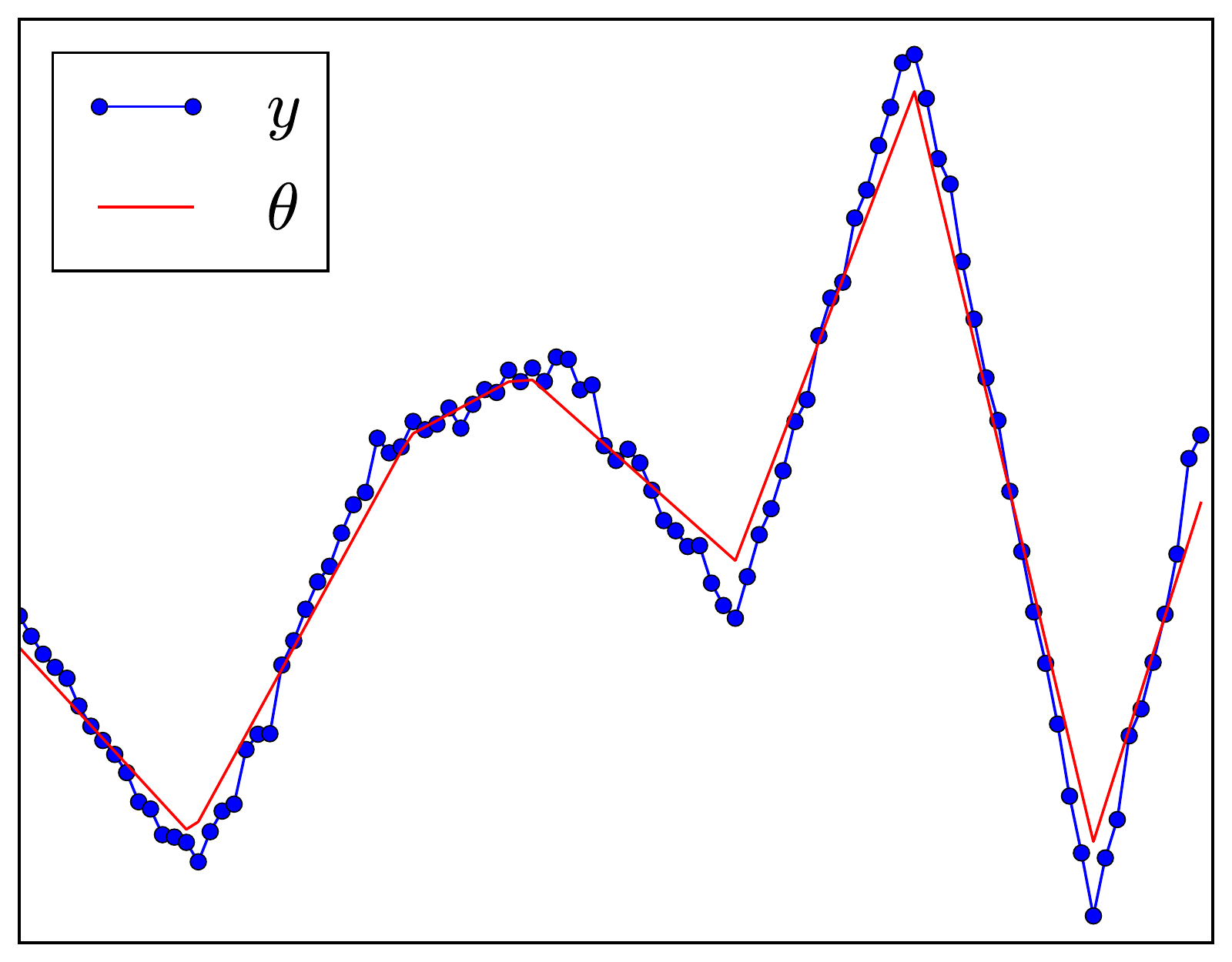}
}
\caption{Trend filtering solutions for different choices of $g$ and $D^{(d,n)}$.}
\label{fig:piecewise}
\end{figure}

As shown in Figure~\ref{fig:piecewise}, when $g = g_{1+}$ the fitted variable $\theta$ has the property of being nearly-monotone and nearly-convex for $d =1$ and $d=2$, respectively. Similarly, when $g=g_1$, the fitted curve would be piecewise constant and piecewise linear for $d =1$ and $d=2$, respectively.  Higher order difference operators may be used, though the first and second order ones are more widely used.

\subsection{A PDAS Framework for Trend Filtering}\label{subsec:tf.pdas}

For brevity, let $D := D^{(d,n)}$.  Denote the optimal solution of problem~\eqref{prob.general} as $(\theta^*,z^*)$.  Corresponding to this optimal solution, we may partition the indices of $D\theta^*$ as follows:
\begin{equation*}
  \mathcal P^* = \{j : (D\theta^*)_j > 0\};\ \ \mathcal N^* = \{j : (D\theta^*)_j < 0\};\ \ \mathcal A^* = \{j : (D\theta^*)_j = 0\}.
\end{equation*}

A typical PDAS framework consists of three steps, as shown in Algorithm~\ref{alg.pdas.frame}: subspace minimization (SSM), termination check, and partition update.  We now discuss each of these steps in turn.

\begin{algorithm}[ht]\small
  \caption{PDAS Framework}
  \label{alg.pdas.frame}
  \begin{algorithmic}[1]
    \State Input an initial partition $(\mathcal P,\mathcal N,\mathcal A)$
        \Loop
        \State Compute the subspace minimizer $(\theta,z)$ corresponding to $(\mathcal P,\mathcal N,\mathcal A)$ \label{step.pdas.frame.ssm}\Comment{subspace minimization (SSM)}
        \State \textbf{if} $(\theta,z)$ is optimal, \textbf{then} terminate and return $(\theta,z)$ \label{step.pdas.frame.check}\Comment{termination check}
        \State Compute a new partition $(\mathcal P,\mathcal N,\mathcal A)$ \label{step.pdas.frame.update} \Comment{partition update}
    \EndLoop
  \end{algorithmic}
\end{algorithm}

\paragraph{Subspace Minimization} Corresponding to an estimate $(\mathcal P,\mathcal N,\mathcal A)$ of the optimal partition $(\mathcal P^*,\mathcal N^*,\mathcal A^*)$ there exists a unique primal-dual estimate $(\theta,z)$ of $(\theta^*,z^*)$.  Denoting $\mathcal I$ as the union $\mathcal P \cup \mathcal N$, the following schematics show the processes for computing $(\theta,z)$ for problem \eqref{prob.general}.

\begin{minipage}[tb]{7cm}
\noindent\fbox{%
\begin{varwidth}{\dimexpr\linewidth-2\fboxsep-2\fboxrule\relax\small}
  SSM for $g(\theta) = \| (D\theta)_+\|_1$
  \begin{algorithmic}[0]
         \State Set 
          $z_j\gets 0 \mbox{ for }j\in\mathcal N \mbox{ and }z_j\gets 1 \mbox{ for }j\in\mathcal P.$
       \State Solve for $(\theta,z_{\mathcal A})$:
      \begin{equation}\label{eq:kkt1}
        \begin{bmatrix}
          I & \lambda D_{\mathcal A}^T\\
          \lambda D_{\mathcal A} & 0
        \end{bmatrix}
        \begin{bmatrix}
          \theta \\ z_{\mathcal A}
        \end{bmatrix}
        =
        \begin{bmatrix}
          y - \lambda D_{\mathcal I}^Tz_{\mathcal I}\\ 0
        \end{bmatrix}.
      \end{equation}
       \State Set
      \begin{equation*}
        \begin{aligned}
          \mathcal V_P &\gets \{j\in\mathcal P: (D\theta)_j < 0\};\\
          \mathcal V_N &\gets \{j\in\mathcal N: (D\theta)_j > 0\};\\
          \mathcal V_{AP} &\gets \{j\in\mathcal A: z_j > 1\};\\
          \mathcal V_{AN} &\gets \{j\in\mathcal A: z_j < 0\}.
        \end{aligned}
      \end{equation*}
     \end{algorithmic}
\end{varwidth}
}
\end{minipage}
\begin{minipage}[tb]{7cm}
\noindent\fbox{%
\begin{varwidth}{\dimexpr\linewidth-2\fboxsep-2\fboxrule\relax\small}
  SSM for $g(\theta) = \| D\theta\|_1$
  \begin{algorithmic}[0]
        \State Set 
          $z_j\gets -1 \mbox{ for }j\in\mathcal N \mbox{ and }z_j\gets 1 \mbox{ for }j\in\mathcal P.$
       \State Solve for $(\theta,z_{\mathcal A})$:
      \begin{equation}\label{eq:kkt0}
        \begin{bmatrix}
          I & \lambda D_{\mathcal A}^T\\
          \lambda D_{\mathcal A} & 0
        \end{bmatrix}
        \begin{bmatrix}
          \theta \\ z_{\mathcal A}
        \end{bmatrix}
        =
        \begin{bmatrix}
          y - \lambda D_{\mathcal I}^Tz_{\mathcal I}\\ 0
        \end{bmatrix}.
      \end{equation}
       \State Set
      \begin{equation*}
        \begin{aligned}
          \mathcal V_P &\gets \{j\in\mathcal P: (D\theta)_j < 0\};\\
          \mathcal V_N &\gets \{j\in\mathcal N: (D\theta)_j > 0\};\\
          \mathcal V_{AP} &\gets \{j\in\mathcal A: z_j > 1\};\\
          \mathcal V_{AN} &\gets \{j\in\mathcal A: z_j < -1\}.
        \end{aligned}
      \end{equation*}
      \end{algorithmic}
      \end{varwidth}
}
\end{minipage}

Note that the solution $(\theta,z_{\mathcal A})$ of system \eqref{eq:kkt1} or \eqref{eq:kkt0} could be efficiently obtained by
\begin{subequations}\label{eq.ssm}
  \begin{align}
    \mbox{solving for }z_{\mathcal A}\mbox{ the system }\qquad& D_{\mathcal A}D_{\mathcal A}^T z_{\mathcal A} = D_{\mathcal A}(y - \lambda D_{\mathcal I}^Tz_{\mathcal I})/\lambda ,  \label{eq.ssm.ls}  \\
    \mbox{then setting }\qquad& \theta \gets y - \lambda D^Tz. \label{eq.ssm.set}
  \end{align}
\end{subequations}
When $D_{\mathcal A}$ is a first (second) order difference matrix, $D_{\mathcal A}D_{\mathcal A}^T$ is a tridiagonal (quindiagonal) matrix, meaning that the left-hand side match in \eqref{eq.ssm.ls} is banded and $z_{\mathcal A}$ can be obtained cheaply.

\paragraph{Termination Check} One can easily show that a computed pair $(\theta,z)$ is optimal if the set $\mathcal V = \mathcal V_P\cup\mathcal V_N\cup\mathcal V_{AP}\cup\mathcal V_{AN}$, consisting of indices of $D\theta$ and $z$ corresponding to violated bounds, is empty \cite{CurtHan14}.  Thus, when $\mathcal V$ is empty, optimality has been reached and the algorithm terminates.  Otherwise, these sets indicate a manner in which the partition could be updated.

\paragraph{Partition Update} In PDAS methods, the indices of $D\theta$ and $z$ violating their bounds are deemed candidates for being switched from one index set in a partition to another.  Specifically, a standard update in a PDAS method \cite{HintItoKuni03} involves the following steps:
\begin{equation}\label{step.pdas.partition}
  \begin{aligned}
    \mathcal P &\gets (\mathcal P\backslash\mathcal V_P)\cup \mathcal V_{AP};\\
    \mathcal N &\gets (\mathcal P\backslash\mathcal V_N)\cup \mathcal V_{AN};\\
    \mathcal A &\gets \mathcal A\backslash (\mathcal V_{AP}\cup\mathcal V_{AN})\cup (\mathcal V_P\cup\mathcal V_N).
  \end{aligned}
\end{equation}

\subsection{Safeguarding}\label{subsec:tf.safeguard}

There is no convergence guarantee for Algorithm~\ref{alg.pdas.frame} for an arbitrary instance of \eqref{prob.general}; indeed, an example illustrating that the method can cycle when solving a BQP is given in \cite{CurtHanRobi14}.  This example is presented here to illustrate that Algorithm~\ref{alg.pdas.frame} can cycle for certain instances of \eqref{prob.general}.  
\begin{example}
Let $y = (603,  996,  502,   19,   56,  139)^T$, $\lambda = 100$, and $g(\theta) = \|D^{(2,6)}\theta \|_1$.  The iterates produced in the first few iterations of Algorithm~\ref{alg.pdas.frame} are shown in Table~\ref{tab:cycle}.
\begin{table}[ht]\footnotesize
  \centering
  \begin{tabular}{|c|l|l|l|l|l|}
    \hline
    Iter & $\mathcal P$ &  $\mathcal N$ &  $\mathcal A$ & $D\theta$ & $z$\\
    \hline
    0 & $\{2,3,4\}$& $\{1\}$   & $\emptyset$                              & $(13,-689,820,-254)^T$ & $(-1,1,1,1)^T$\\
    1 &$\{3\}$                     & $\emptyset$       &$\{1,2,4\}$             & $(0,0,\frac{4227}{38},0)^T$& $(-\frac{5293}{2280},-\frac{482}{475},1,\frac{5201}{5700})^T$\\
    2 & $\{3\}$            & $\{1,2\}$ & $\{4\}$                                  & $(-787,520,-16,0)^T$ & $(-1,-1,1,\frac{91}{100})^T$\\
    3 & $\emptyset$                & $\{1\}$           &  $\{2,3,4\}$  & $(-\frac{887}{5},0,0,0)^T$ & $(-1,\frac{127}{125},\frac{371}{125},\frac{943}{500})^T$\\
    4 & $\{2,3,4\}$& $\{1\}$   & $\emptyset$                              & $(13,-689,820,-254)^T$ & $(-1,1,1,1)^T$\\
    \hline
    \multicolumn{6}{|c|}{$\vdots$}\\
    \hline
  \end{tabular}
  \caption{An illustration of Algorithm~3 cycling}
  \label{tab:cycle}
\end{table}
Since the algorithm returns to a previously explored partition without computing an optimal solution, the algorithm cycles, i.e., it is not convergent for this problem instance from the given starting point.
\end{example}

A simple safeguarding strategy to overcome this issue and ensure convergence is proposed in \cite{kim10} and subsequently embedded in the work of \cite{ByrdChinNoceOzto12}.  In particular, when $|\mathcal V|$ fails to decrease for several consecutive iterations, a backup procedure is invoked in which \eqref{step.pdas.partition} is modified to only change partition membership of one index of $\mathcal V$.  We present this approach in Algorithm~\ref{alg.pdas.murty}.

\begin{algorithm}[ht]\small
  \caption{PDAS Framework with Safeguarding \cite{ByrdChinNoceOzto12}}
  \label{alg.pdas.murty}
  \begin{algorithmic}[1]
    \State Input $(\mathcal P,\mathcal N,\mathcal A)$ and an integer $\tt t_{max}$
    \State Initialize $\tt V_{best} \gets \infty$ and $t \gets 0$
    \Loop
      \State Compute the subspace minimizer $(\theta,z)$ corresponding to $(\mathcal P,\mathcal N,\mathcal A)$ 
      \State \textbf{if} $|\mathcal V| = 0$, \textbf{then} terminate and return $(\theta,z)$
      \If{$|\mathcal V| < \tt V_{best}$} 
      \State Set $t\gets 0$ and $\tt V_{best}\gets$ $|\mathcal V|$
      \State Apply \eqref{step.pdas.partition}
      \ElsIf{$|\mathcal V| \geq \tt V_{best}$}
      \State Set $t\gets t+1$
      \If{$t \leq \tt t_{max}$}
      \State Apply \eqref{step.pdas.partition}
      \ElsIf{$t > \tt t_{max}$}
      \State Set $j \gets \min\{i: i \in \mathcal V\}$ and apply a partition update by
      \begin{equation}\label{eq:safeguard.murty}
      \begin{aligned}
        &\mbox{moving } j \mbox{ from } \Pcal \mbox{ to } \Acal , \mbox{ if } j\in\Vcal_P\\
        &\mbox{moving } j \mbox{ from } \Ncal \mbox{ to } \Acal , \mbox{ if } j\in\Vcal_N\\
        &\mbox{moving } j \mbox{ from } \Acal \mbox{ to } \Pcal , \mbox{ if } j\in\Vcal_{AP}\\
        &\mbox{moving } j \mbox{ from } \Acal \mbox{ to } \Ncal , \mbox{ if } j\in\Vcal_{AN}
      \end{aligned}
      \end{equation}
      \EndIf
      \EndIf
      \EndLoop
  \end{algorithmic}
\end{algorithm}

The safeguard of Algorithm~\ref{alg.pdas.murty} employs a heuristic to decide whether to update the partition by \eqref{step.pdas.partition} or \eqref{eq:safeguard.murty}.  However, we have found an alternative strategy that performs better in our experiments. First, unlike that of \cite{kim10,ByrdChinNoceOzto12}, our safeguard changes the memberships of a portion of $\mathcal V$, where the portion size is dynamically updated. Another difference in the safeguard design is that we employ a finite queue (first-in-first-out) to store recent values of $|\mathcal V|$ of which the maximum serves as the reference measure that diminishes as the algorithm continues.  When an element is pushed into the queue that is full, the earliest element is removed.  We present our approach in Algorithm~\ref{alg.pdas.safeguard}.

\begin{algorithm}[ht]\small
  \caption{PDAS Framework with Safeguarding}
  \label{alg.pdas.safeguard}
  \begin{algorithmic}[1]
    \State Input $(\mathcal P,\mathcal N,\mathcal A)$, queue $Q_m$ with size $m$, proportion $p \in (0,1]$, parameter $\delta_s\in (0,1)$ and $\delta_e\in (1,\infty)$
    \Loop
      \State Compute the subspace minimizer $(\theta,z)$ corresponding to $(\mathcal P,\mathcal N,\mathcal A)$ 
      \State \textbf{if} $|\mathcal V| = 0$, \textbf{then} terminate and return $(\theta,z)$
      \State Set $\tt max/min\gets$ maximum/minimum of $Q_m$
      \If{$|\mathcal V| > \tt max$} 
      \State set $p \gets \max(\delta_s p,\frac{1}{|\mathcal V|})$ 
      \ElsIf{$|\mathcal V| < \tt min$}
      \State push $|\mathcal V|$ into $Q_m$ and set $p \gets \max(\delta_e p,1)$ 
      \Else 
      \State push $|\mathcal V|$ into $Q_m$
      \EndIf
      \State Sort $\mathcal V$ by $\max (\lambda|D\theta|,|z|)$ and apply \eqref{step.pdas.partition}, only changing the top $p|\mathcal V|$ indices 
      \EndLoop
  \end{algorithmic}
\end{algorithm}

Algorithm~\ref{alg.pdas.safeguard} can be understood as follows.  If $|\mathcal V| = 0$, then $(\theta,z)$ is optimal and the algorithm terminates.  Otherwise, the update \eqref{step.pdas.partition} is to be applied using only the $\lfloor p|\mathcal V| \rfloor$ indices from $\mathcal V$ corresponding to the largest violations.  If $p=1/|\mathcal V|$, then this corresponds to moving only one index as in \cite{kim10}, but if $p \in (1/|\mathcal V|,1]$, then a higher portion of violated indices may be moved.  As long as the reference value---i.e., the maximum of the values in the queue---decreases, the value for $p$ is maintained or is increased.  However, if the reference value fails to decrease, then $p$ is decreased.  Overall, since the procedure guarantees that the reference value is monotonically decreasing and that $p$ is sufficiently reduced whenever a new value for $|\mathcal V|$ is not below the reference value, our strategy preserves the convergence guarantees established in \cite{kim10} while yielding better results in our experiments.

\section{Experiments}\label{sec:experiments}
 We implemented Algorithms~\ref{alg.pdas.practical}, \ref{alg.pdas.frame}, \ref{alg.pdas.murty}, and \ref{alg.pdas.safeguard} in Python 2.7, using the Numpy (version 1.8.2) and Scipy (version 0.14.0) packages for matrix operations.  In the following subsections, we discuss the results of numerical experiments for solving randomly generated instances of problems \eqref{prob.isoreg} and \eqref{prob.general}.  For problem \eqref{prob.isoreg}, merge operations for Algorithm~\ref{alg.pdas.practical} were carried out sequentially (but recall Figure~\ref{fig:pdas.merge} in which we have illustrated that they could be carried out in parallel).  Throughout our experiments, we set $\omega$ as an all-one vector.

\subsection{Test on Isotonic Regression}\label{sec:numeric.isoreg}

We compare the performance of Algorithm~\ref{alg.pdas.practical} (\verb|PDAS|) with a Python implementation of Algorithm~\ref{alg.pava} (\verb|PAV|) that was later integrated into scikit-learn (version 0.13.1) \cite{scikit-learn}. The data $y_i$ are generated by $y_i \gets i + \varepsilon_i$ where $\varepsilon_i \sim \mathcal N(0,4)$.  The initial partition is set as $J_0 = \{\{1\},\{2\},\ldots,\{n\}\}$ for both algorithms.  We generated 10 random instances each for $n\in\{1\times 10^4,5\times 10^4,\ldots,33\times 10^4\}$.  Boxplots for running time in seconds ($\tt Time\ (s)$) and number of merge operations ($\tt \#\ Merge$) are reported in Figure~\ref{fig:ir.cold}.
\begin{figure}[ht]
\begin{center}
  \includegraphics[width = 2.6in]{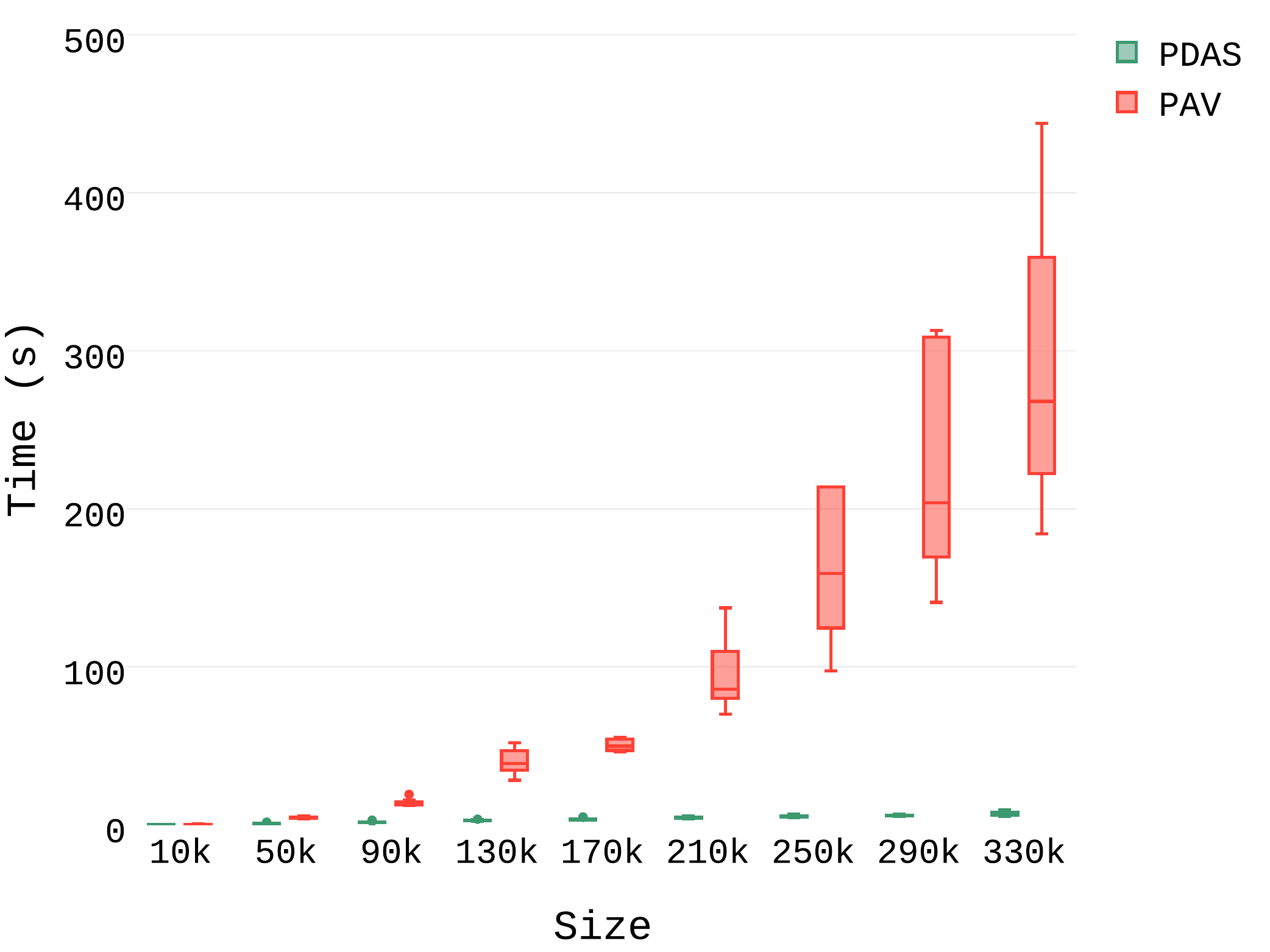}
  \includegraphics[width = 2.6in]{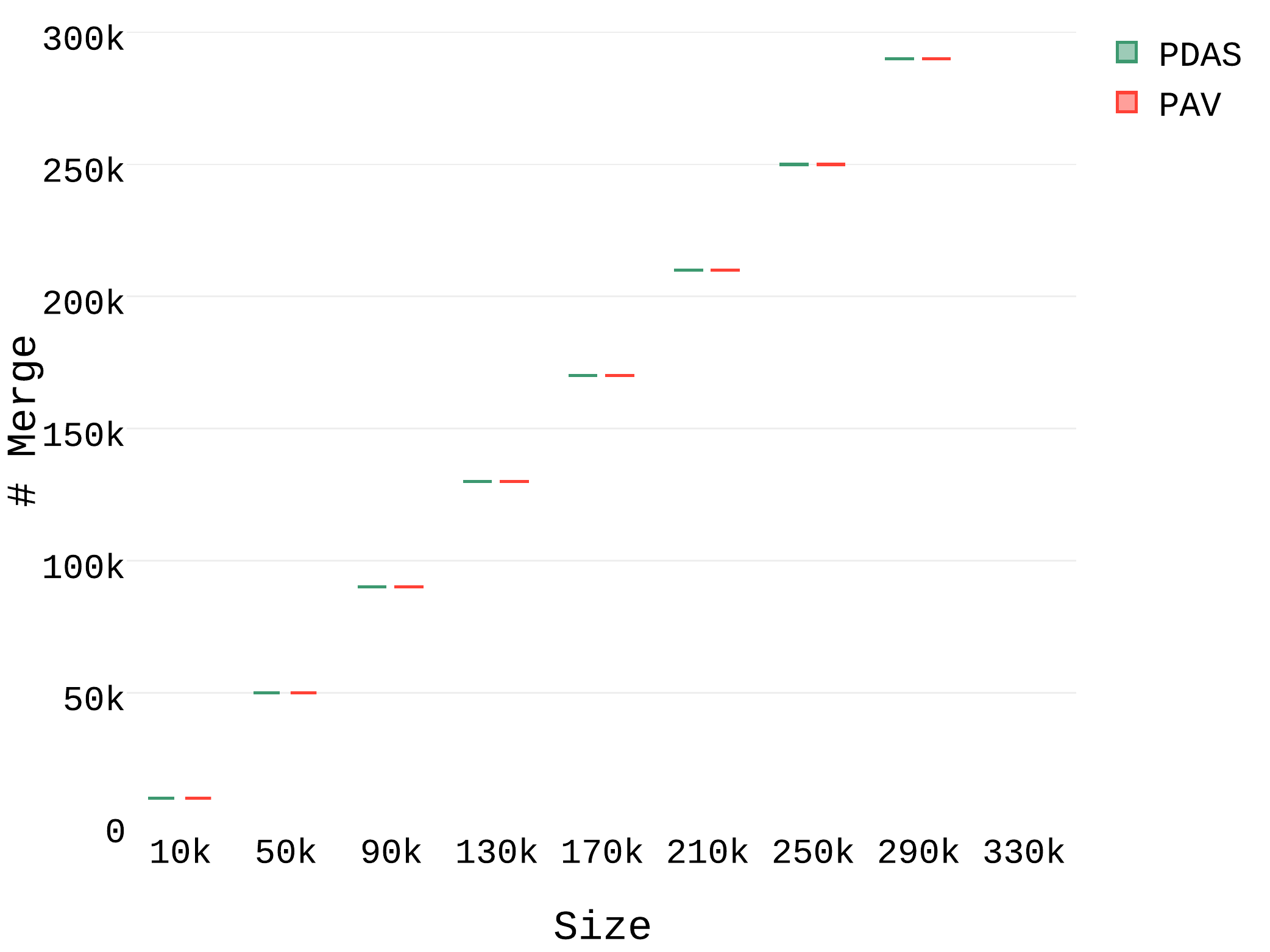}
\end{center}
\caption{Comparison of PDAS and PAV in running time (left) and \# of merges (right).}
  \label{fig:ir.cold}
\end{figure}

Figure~\ref{fig:ir.cold} demonstrates that our implementation of \verb|PDAS| outperforms \verb|PAV| in terms of running time.  That being said, when both use the set of singletons as the starting point, the numbers of merge operations performed by the two algorithms are nearly identical.

\paragraph{Warm-starting}  Figure~\ref{fig:ir.cold} does not show an obvious advantage of \verb|PDAS| over \verb|PAV| in terms of the numbers of merge operations.  However, as claimed, we now show an advantage of \verb|PDAS| in terms of its ability to exploit a good initial partition.  We simulate warm-starting \verb|PDAS| by generating an instance of \eqref{prob.isoreg} as in our previous experiment, solving it with \verb|PDAS|, and using the solution as the starting point for solving related instances for which the data vector $y$ has been perturbed.  In particular, for each problem size $n$, we generated 10 perturbed instances by adding a random variable $\epsilon_i \sim \mathcal N (0,10^{-2})$ to each $y_i$.  The results of solving these instances are reported in Figures~\ref{fig:warm} and \ref{fig:warm2}.
\begin{figure}[ht]
\begin{center}
  \includegraphics[width = 2.6in]{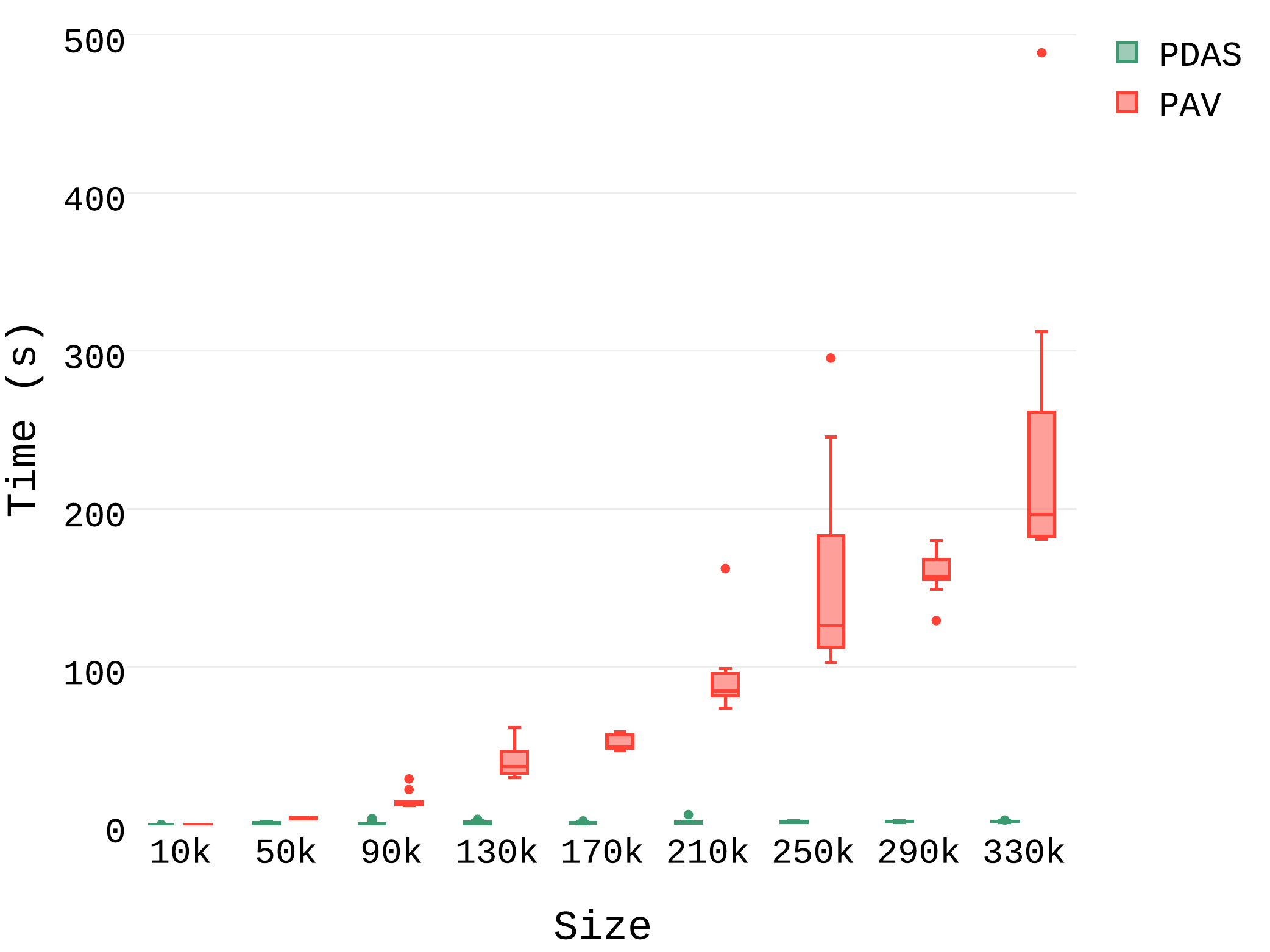}
  \includegraphics[width = 2.6in]{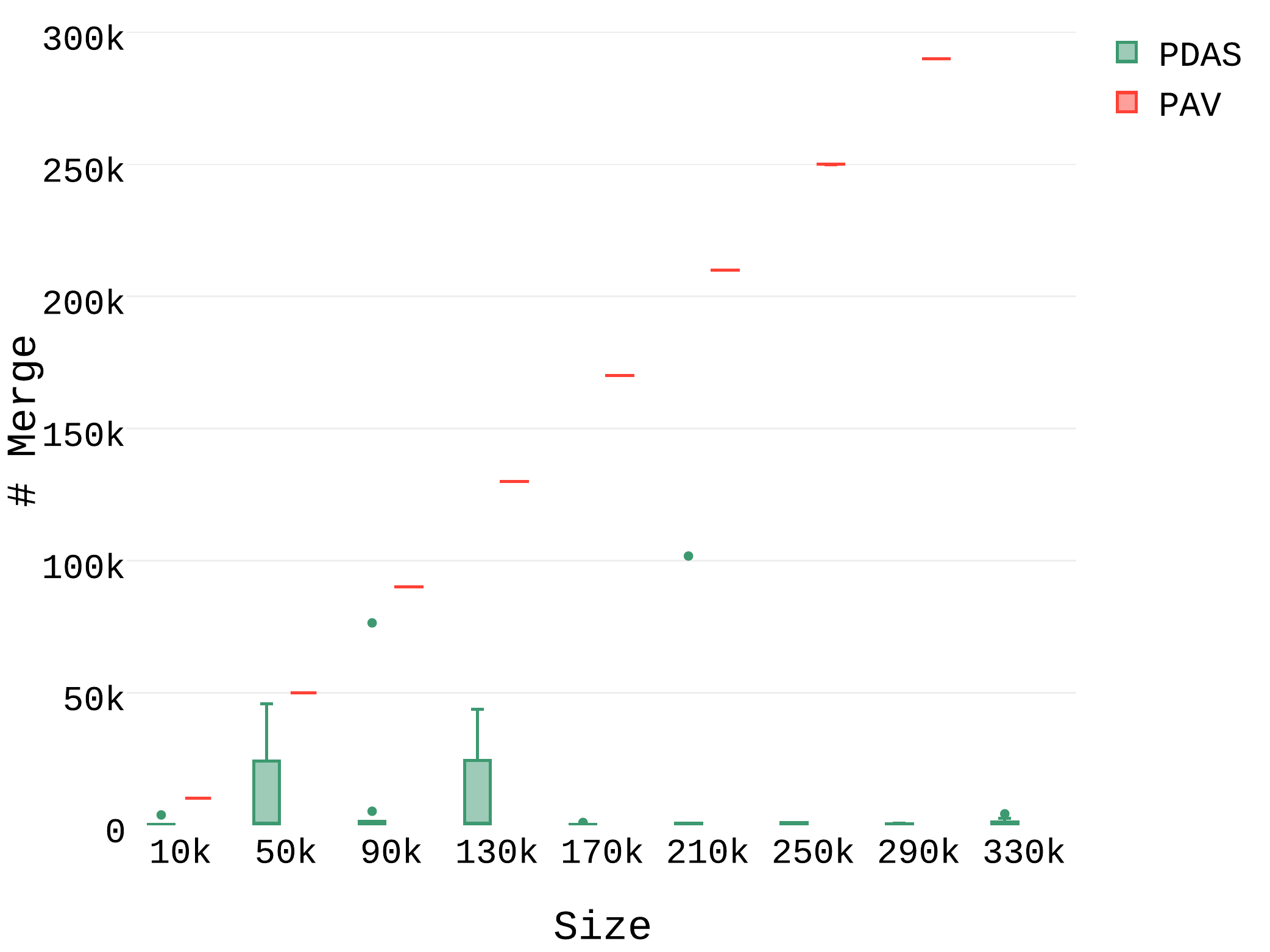}
\end{center}
\caption{Comparison of warm-started PDAS and PAV in running time (left) and \# of merges (right).}
  \label{fig:warm}
\end{figure}

\begin{figure}[ht]
\begin{center}
  \includegraphics[width = 2.6in]{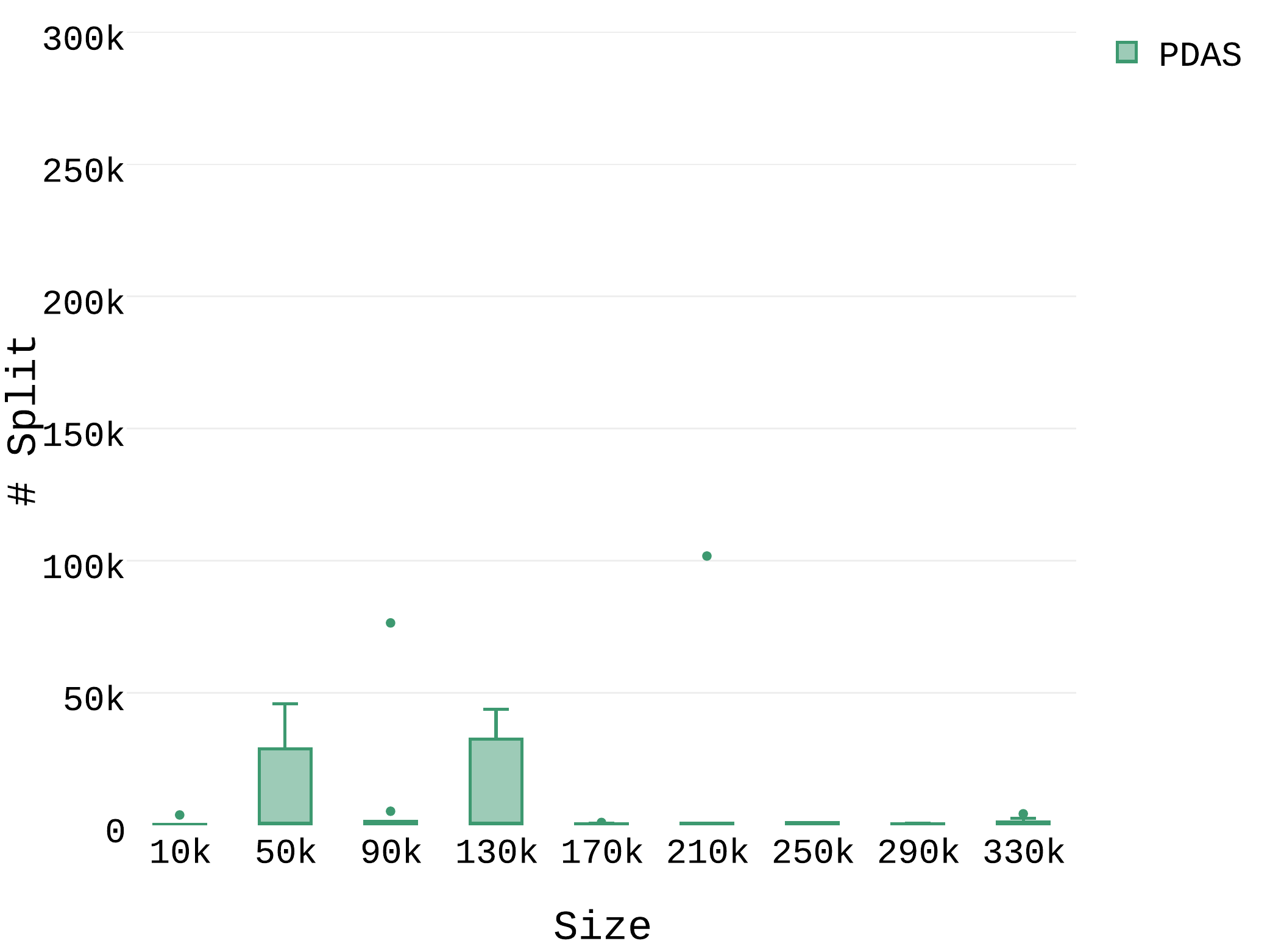}
\end{center}
\caption{Summary of warm-started PDAS in terms of \# of splits.}
  \label{fig:warm2}
\end{figure}

Since \verb|PAV| is not able to utilize a good initial partition, the required work (i.e., number of merge operations) for solving the perturbed problems is not cheaper than for the base instance.  In contrast, warm-starting the \verb|PDAS| algorithm can greatly reduce the computational cost as observed in the much-reduced number of merge operations (even after accounting for the added split operations).



\subsection{Test on Trend Filtering}

We now compare the performance of several PDAS variants for trend filtering.  In particular, we compare the straightforward PDAS method of Algorithm~\ref{alg.pdas.frame} (\verb|PDAS|), Algorithm~\ref{alg.pdas.murty} (\verb|SF1|), and the PDAS method with our proposed safeguard strategy in Algorithm~\ref{alg.pdas.safeguard} (\verb|SF2|).  The safeguard parameter ${\tt t_{max}} = 5$ was chosen for \verb|SF1| and we similarly set $m=5$, $\delta_s = 0.9$, and $\delta_e = 1.1$ for \verb|SF2|.  We generated 10 random problem instances each for $n\in\{10^4,1.7\times 10^5, 3.3\times 10^5\}$ where, for each instance, the data vector had $y_i$ uniformly distributed in $[0,10]$.  Such datasets had minimum pattern and thus made each problem relatively difficult to solve.  We considered both regularization functions $g_1$ and $g_{1+}$ defined in \S\ref{subsec:tf.reg} with difference matrices $D^{(1,n)}$ and $D^{(2,n)}$, setting $\lambda = 10$ in all cases.  For all runs, we set an iteration limit of 800; if an algorithm failed to produce the optimal solution within this limit, then the run was considered a failure.  The percentages of successful runs for each algorithm is reported in Table~\ref{tab.percent.success}.
\begin{table}[ht]\scriptsize
  \caption{Percentages of successful runs for each algorithm and problem type}
  \label{tab.percent.success}
  \centering
  {\tt
  \begin{tabular}{|r|rrr|rrr|rrr|rrr|}
    \hline
    \multirow{2}{*}{$n$ (size)}&\multicolumn{12}{|c|}{\% of success}\\
    \cline{2-13}
     & \multicolumn{3}{|c|}{$g(\theta) = \|(D^{(1,n)}\theta)_+\|_1$} & \multicolumn{3}{|c|}{$g(\theta) = \|D^{(1,n)}\theta\|_1$} & \multicolumn{3}{|c|}{$g(\theta) = \|(D^{(2,n)}\theta)_+\|_1$}  & \multicolumn{3}{c|}{$g(\theta) = \|D^{(2,n)}\theta\|_1$}\\
     & PDAS & SF1 & SF2& PDAS & SF1 & SF2& PDAS & SF1 & SF2& PDAS & SF1 & SF2\\
    \hline

1.0e+4  & 1.0 & 1.0 & 1.0 & 1.0 & 1.0 & 1.0 & 0.0 & 1.0 & 1.0 & 0.0 & 1.0 & 1.0 \\
1.7e+5  & 1.0 & 1.0 & 1.0 & 1.0 & 1.0 & 1.0 & 0.0 & 0.0 & 1.0 & 0.0 & 0.0 & 1.0 \\
3.3e+5  & 1.0 & 1.0 & 1.0 & 1.0 & 1.0 & 1.0 & 0.0 & 0.0 & 1.0 & 0.0 & 0.0 & 1.0 \\
    \hline
  \end{tabular}
  }
\end{table}

We observe from Table~\ref{tab.percent.success} that all algorithms solved all instances when the regularization function involved a first-order difference matrix, but that \verb|PDAS| and \verb|SF1| both had failures when a second-order difference matrix is used.  By contrast, our proposed safeguard in \verb|SF2| results in a method that is able to solve all instances within the iteration limit.  This shows that our proposed safeguard, which allows more aggressive updates, can be more effective than a conservative safeguard.

To compare further the performance of the algorithms, we collected the running time and iteration number for all successfully solved instances. Figure~\ref{fig:tf1.cold} demonstrates that when $D = D^{(1,n)}$, all algorithms show very similar performance.  However, when $D = D^{(2,n)}$ as in Figure~\ref{fig:tf2.cold}, the results show that \verb|SF2| is not only more reliable than \verb|PDAS| and \verb|SF1|; it is also more efficient even when \verb|SF1| is successful.  We also include the results for the interior-point method (\verb|IPM|) proposed in \cite{kim09}, but emphasize that this algorithm is implemented in Matlab (as opposed to Python) and is only set up to solve the instances when an $\ell_1$-regularization function is used. Although the interior point method demonstrates impressive performance, we remark that in general it is difficult to warm-starting interior point methods \cite{john2008implementation}, despite recent efforts toward this direction \cite{yildirim2002warm,gondzio1996computational,gondzio2002reoptimization} .

\begin{figure}[ht]
\centering
\subfigure[$D= D^{(1,n)}$, $g = g_{1+}$]{
\includegraphics[width = 2.1in]{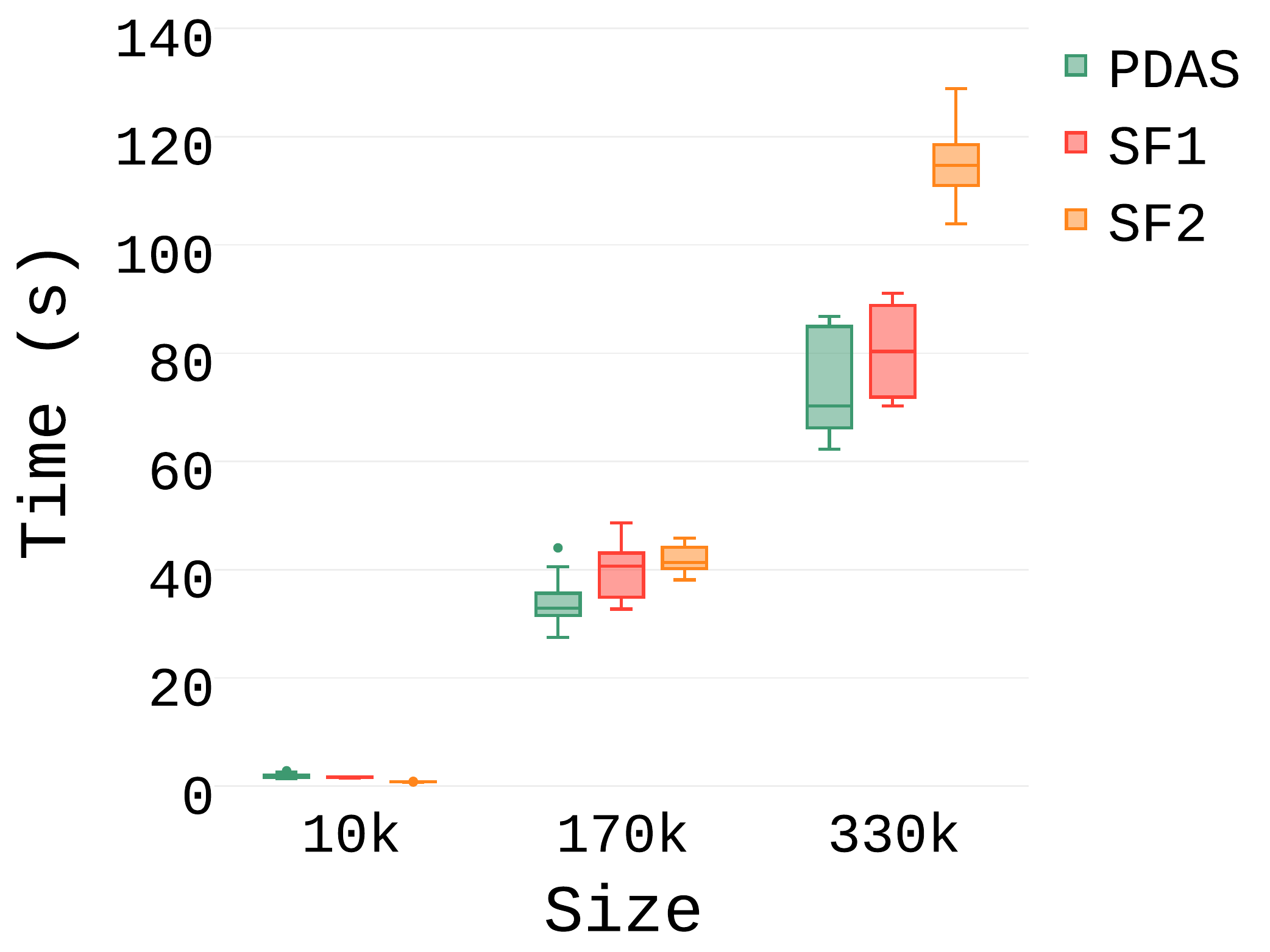}
\includegraphics[width = 2.1in]{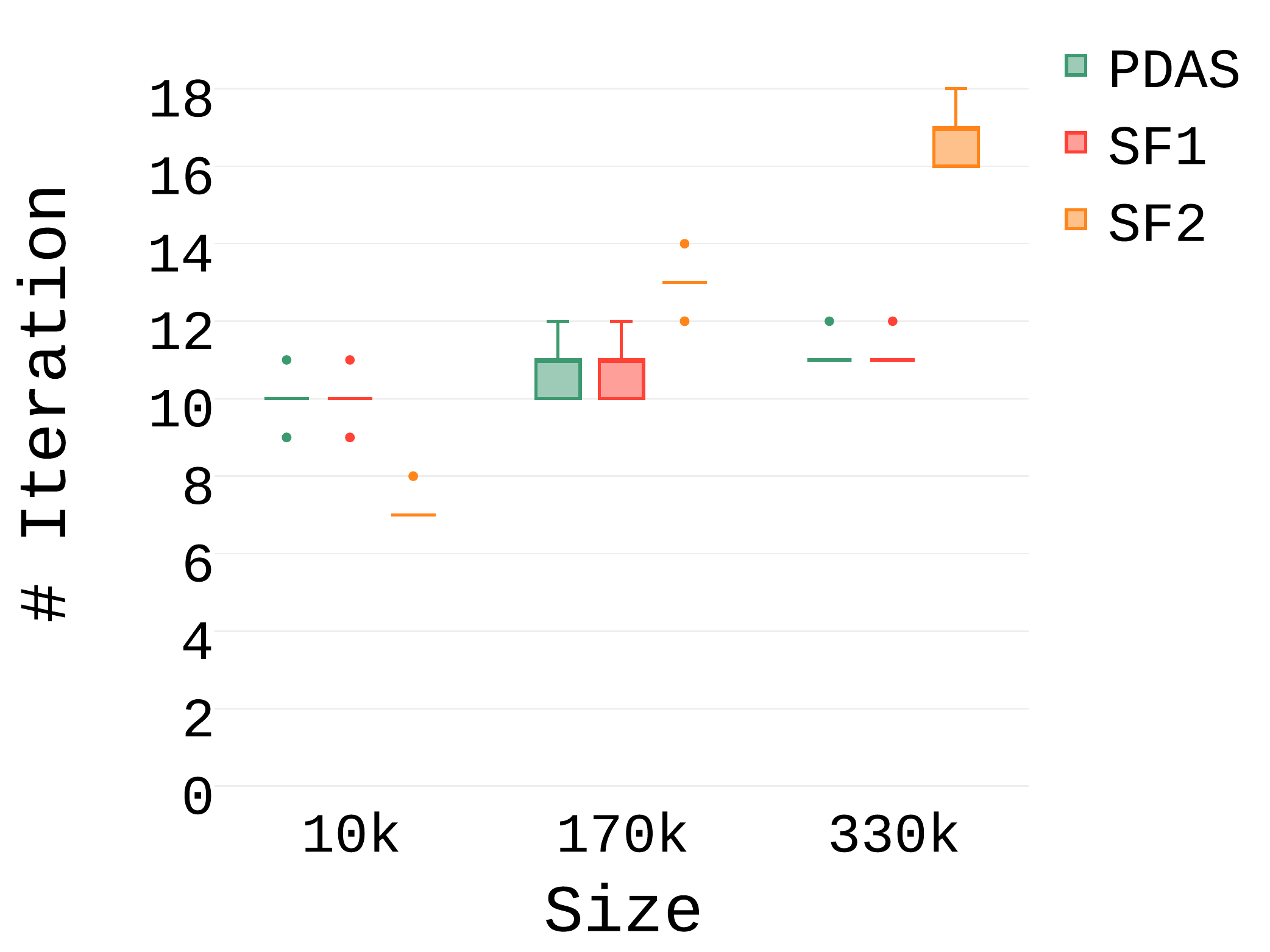}
 }
\subfigure[$D= D^{(1,n)}$, $g = g_1$]{
\includegraphics[width = 2.1in]{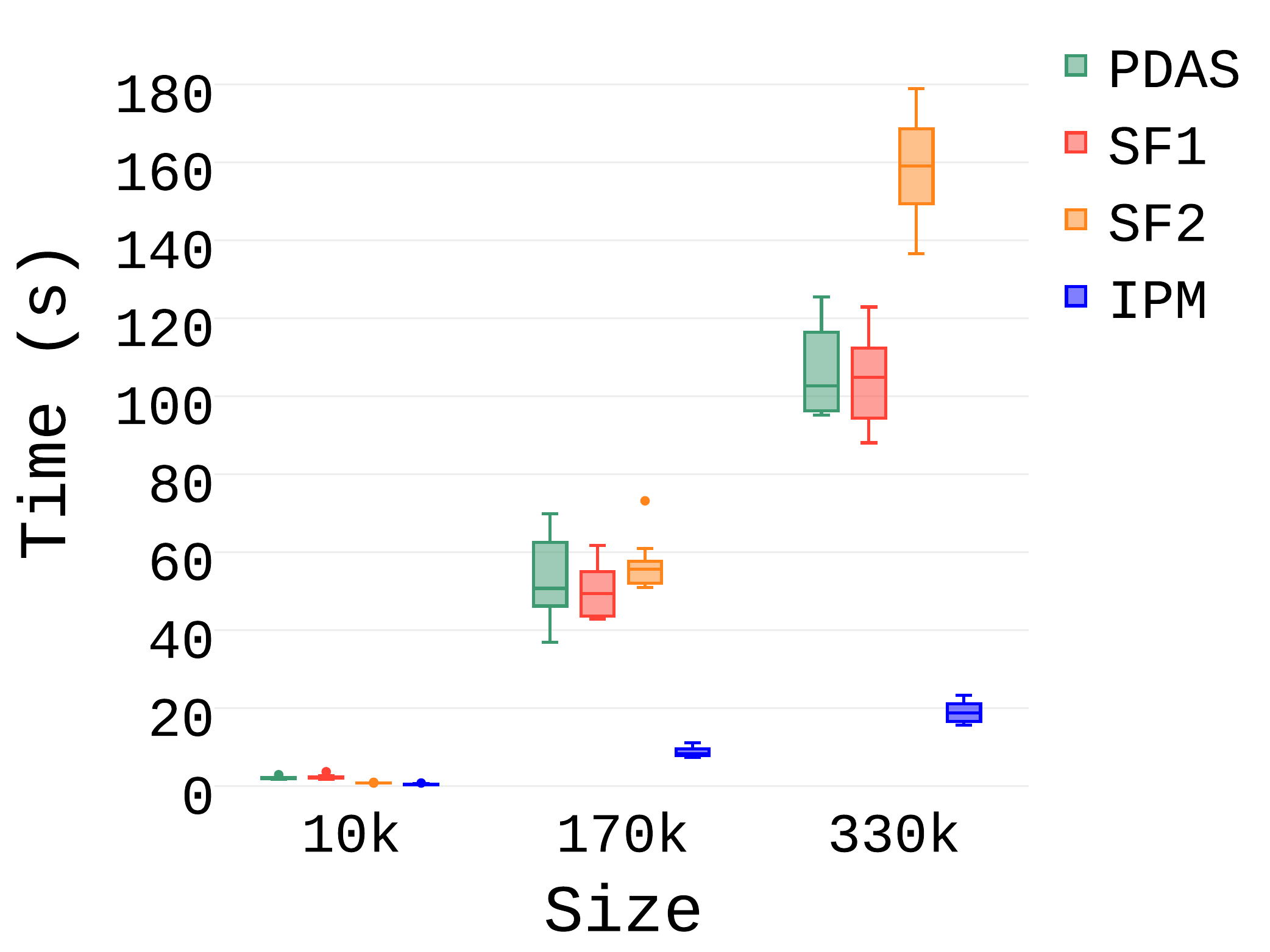}
\includegraphics[width = 2.1in]{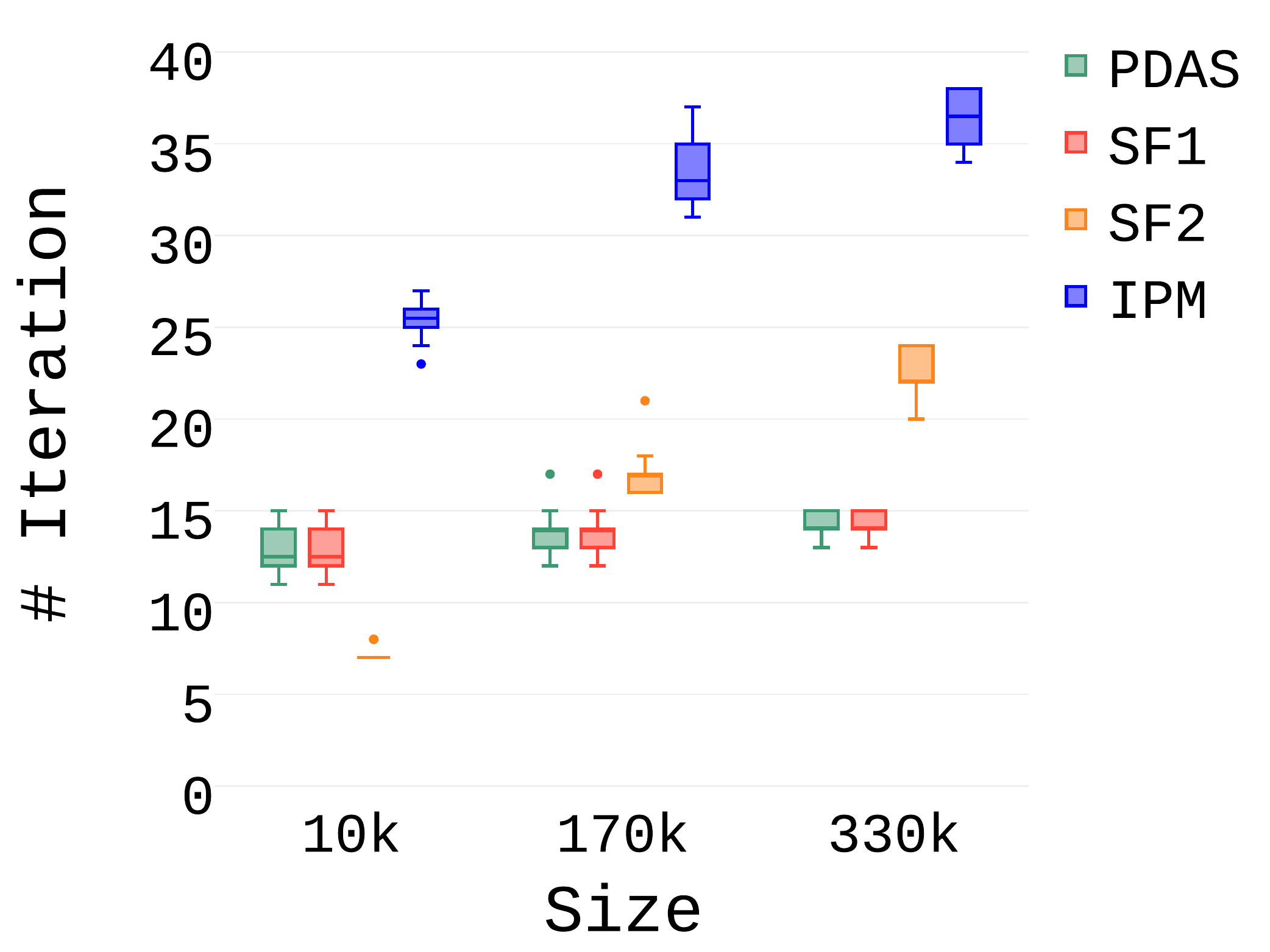}
}
\caption{PDAS vs IPM for $D = D^{1,n}$ and different choices of $g$.}
  \label{fig:tf1.cold}
\end{figure}

\begin{figure}[ht]
\centering
\subfigure[$D= D^{(2,n)}$, $g = g_{1+}$]{
\includegraphics[width = 2.1in]{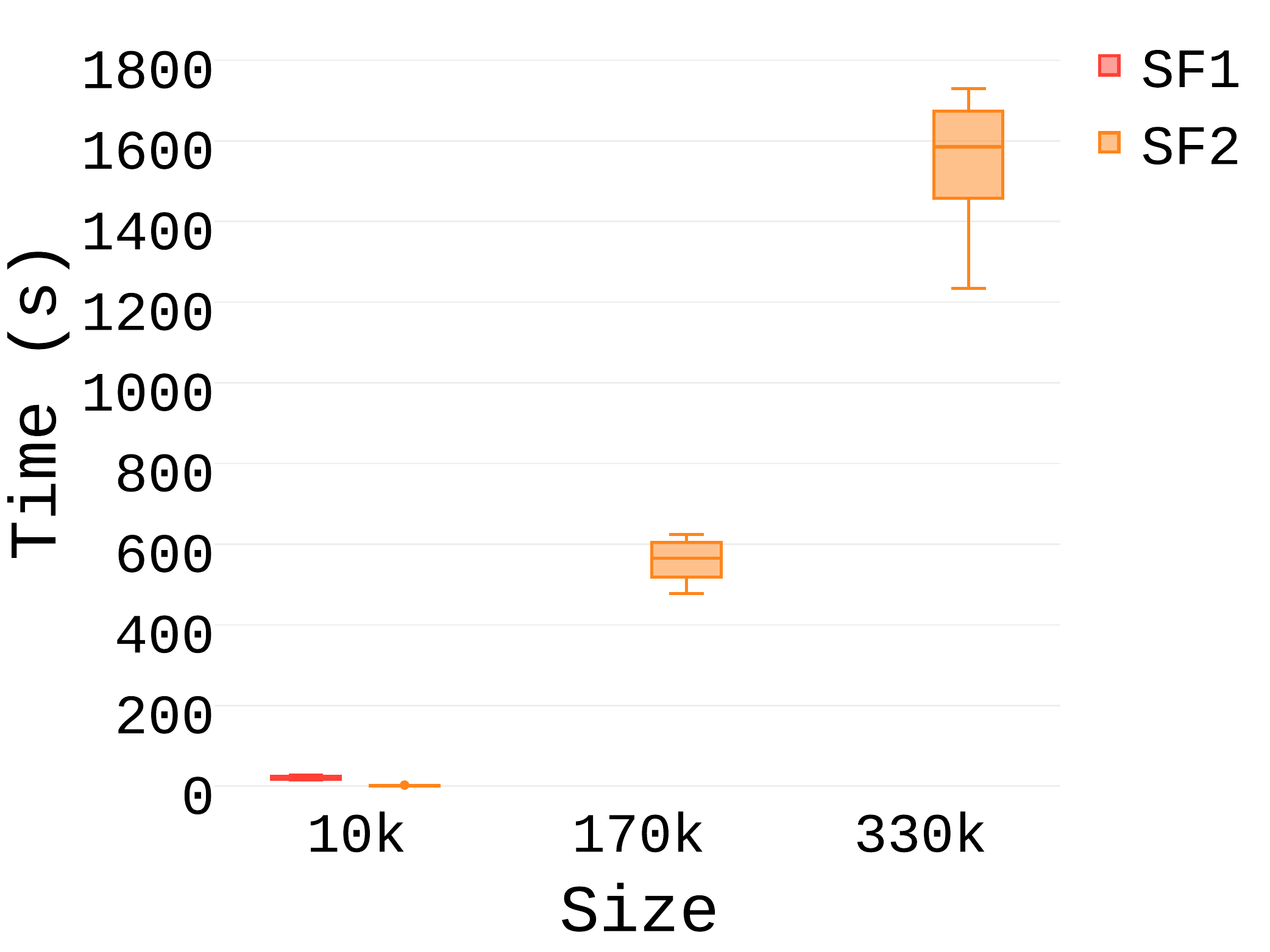}
\includegraphics[width = 2.1in]{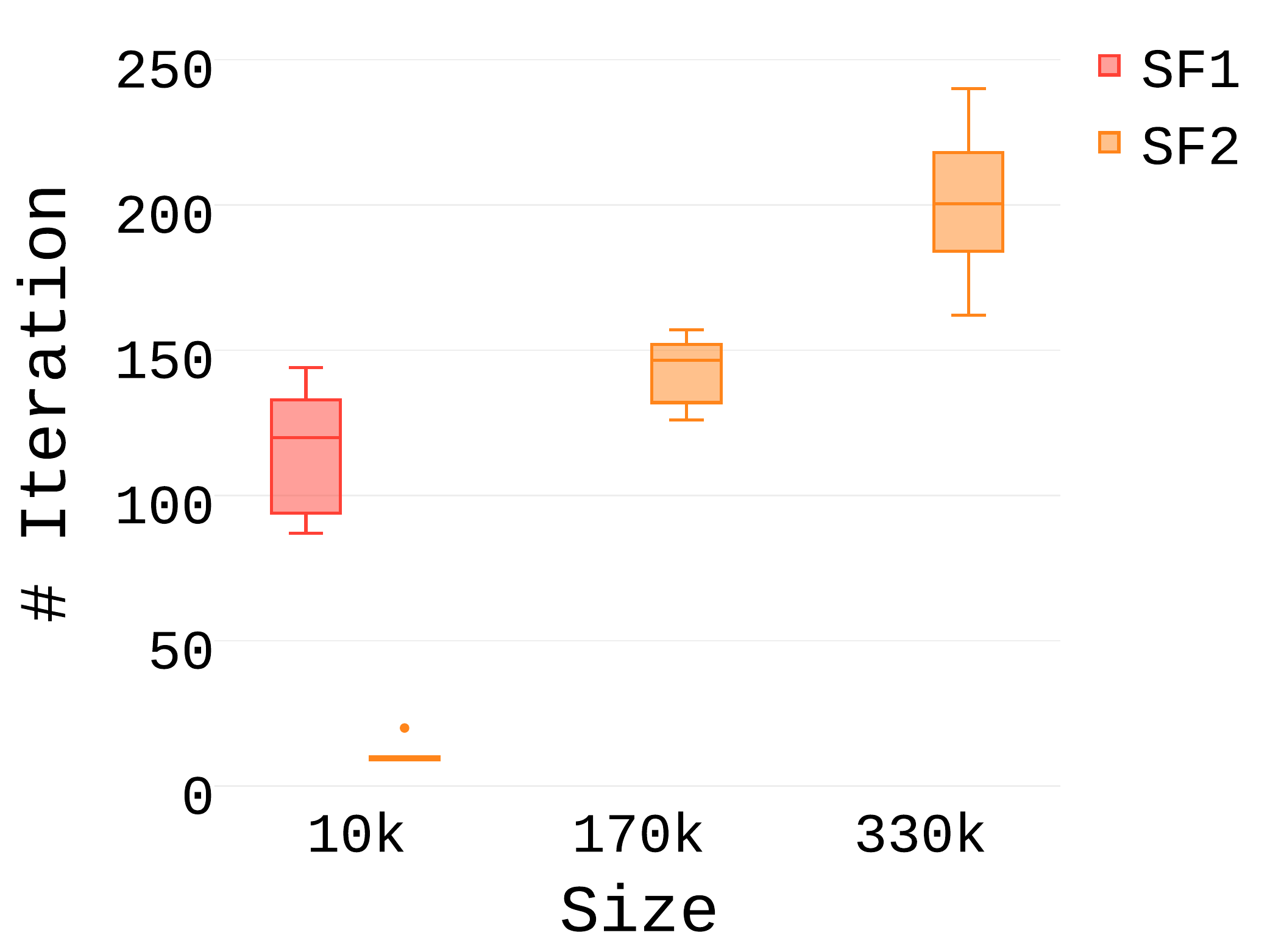}
}
\subfigure[$D= D^{(2,n)}$, $g = g_1$]{
\includegraphics[width = 2.1in]{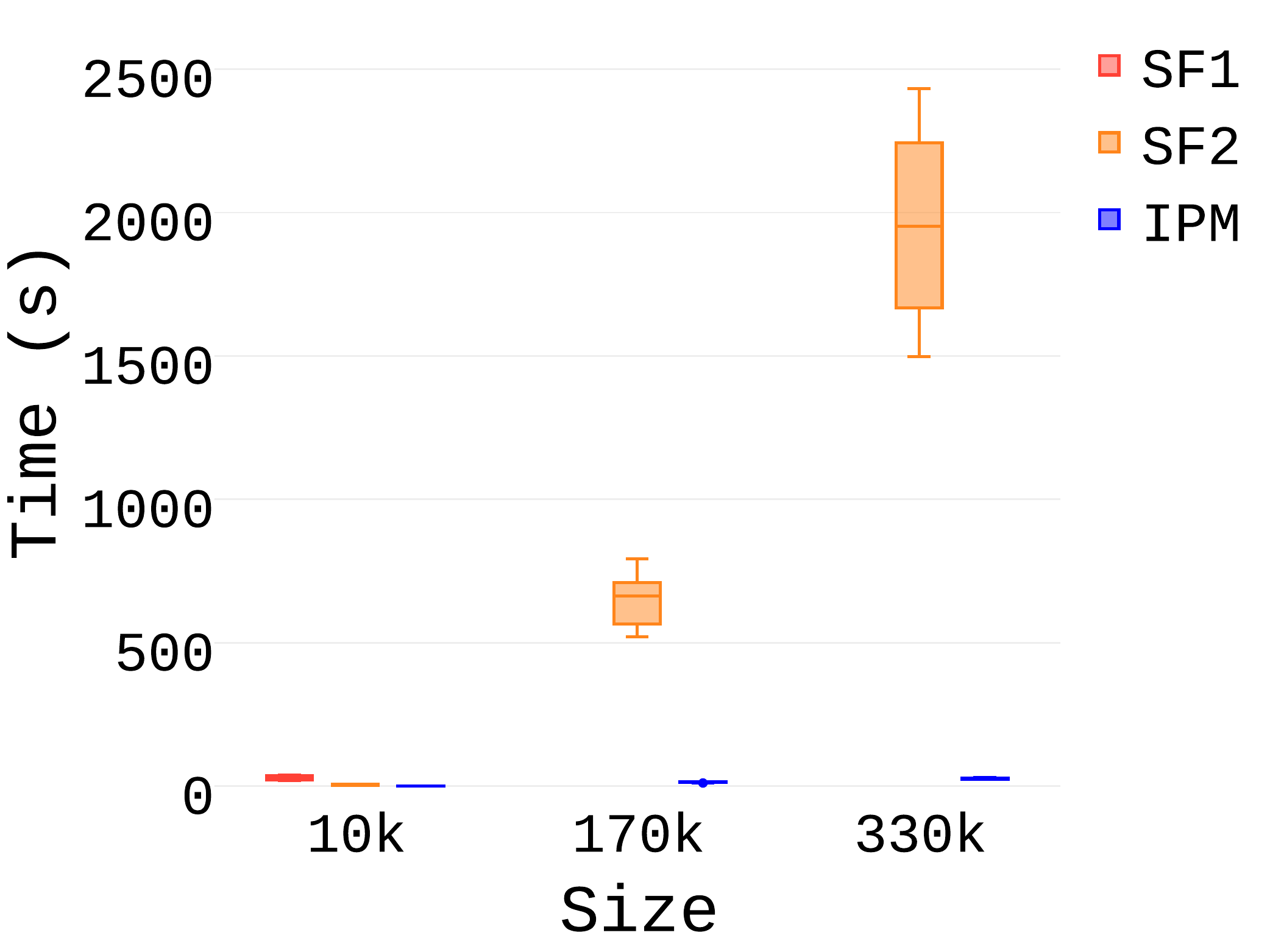}
\includegraphics[width = 2.1in]{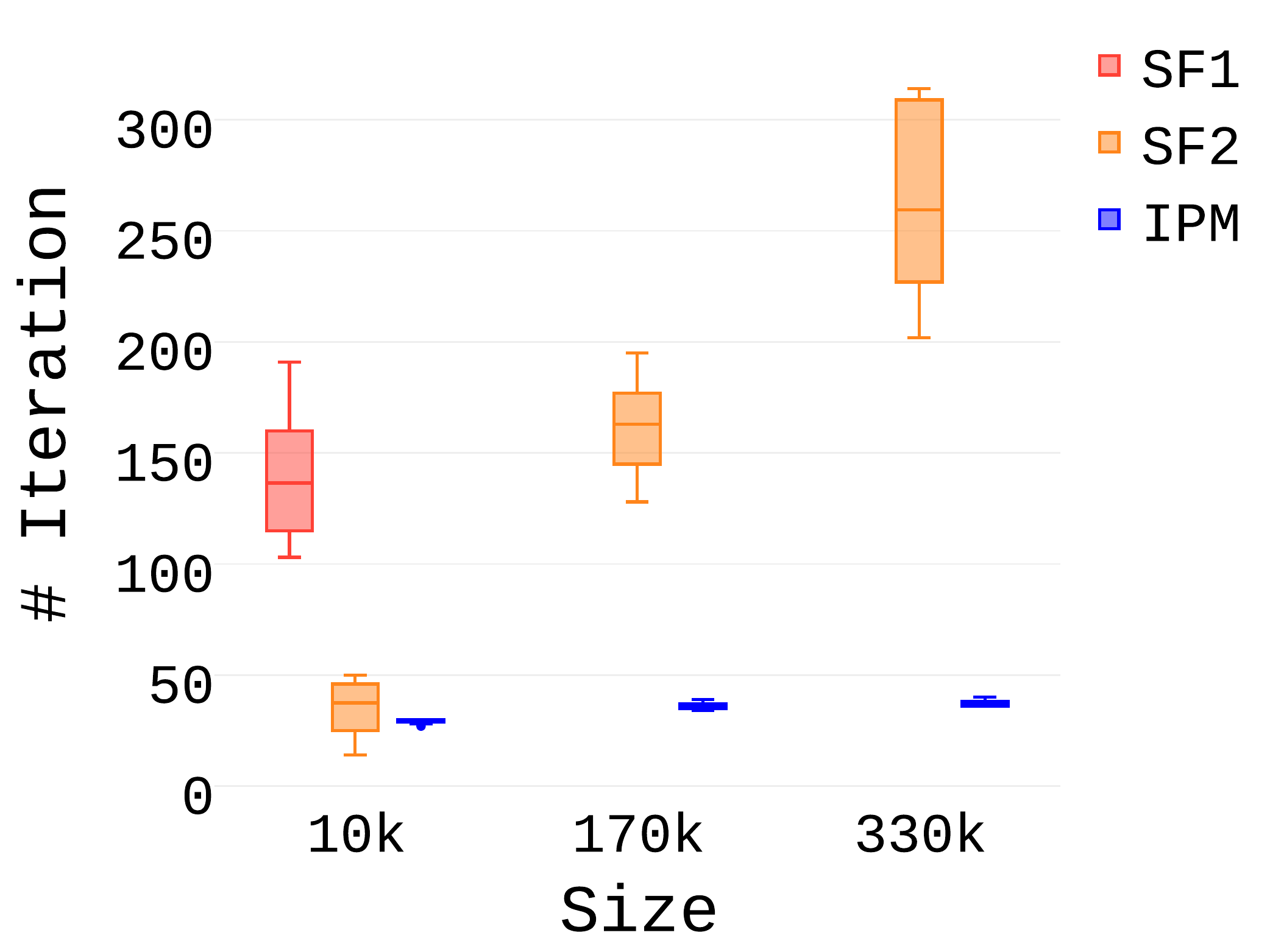}
}
\caption{PDAS vs IPM for $D = D^{2,n}$ and different choices of $g$.}
  \label{fig:tf2.cold}
\end{figure}


\subsubsection*{Warm-starting}  We conclude our experiments by comparing the performance of \verb|IPM|---which is not set up for warm-starting---and warm-started \verb|SF2|.  As in \S\ref{sec:numeric.isoreg}, we generated 10 perturbed instances for a given dataset by adding a random variable $\epsilon_i \sim\mathcal N (0,10^{-2})$ to $y_i$.  The running times and numbers of iterations for solving the perturbed problems are reported via boxplots in Figure~\ref{fig:tf.warm}.

\begin{figure}[ht]
\centering
\subfigure[$D = D^{(1,n)}$, $g = g_1$]{
\includegraphics[width = 2.2in]{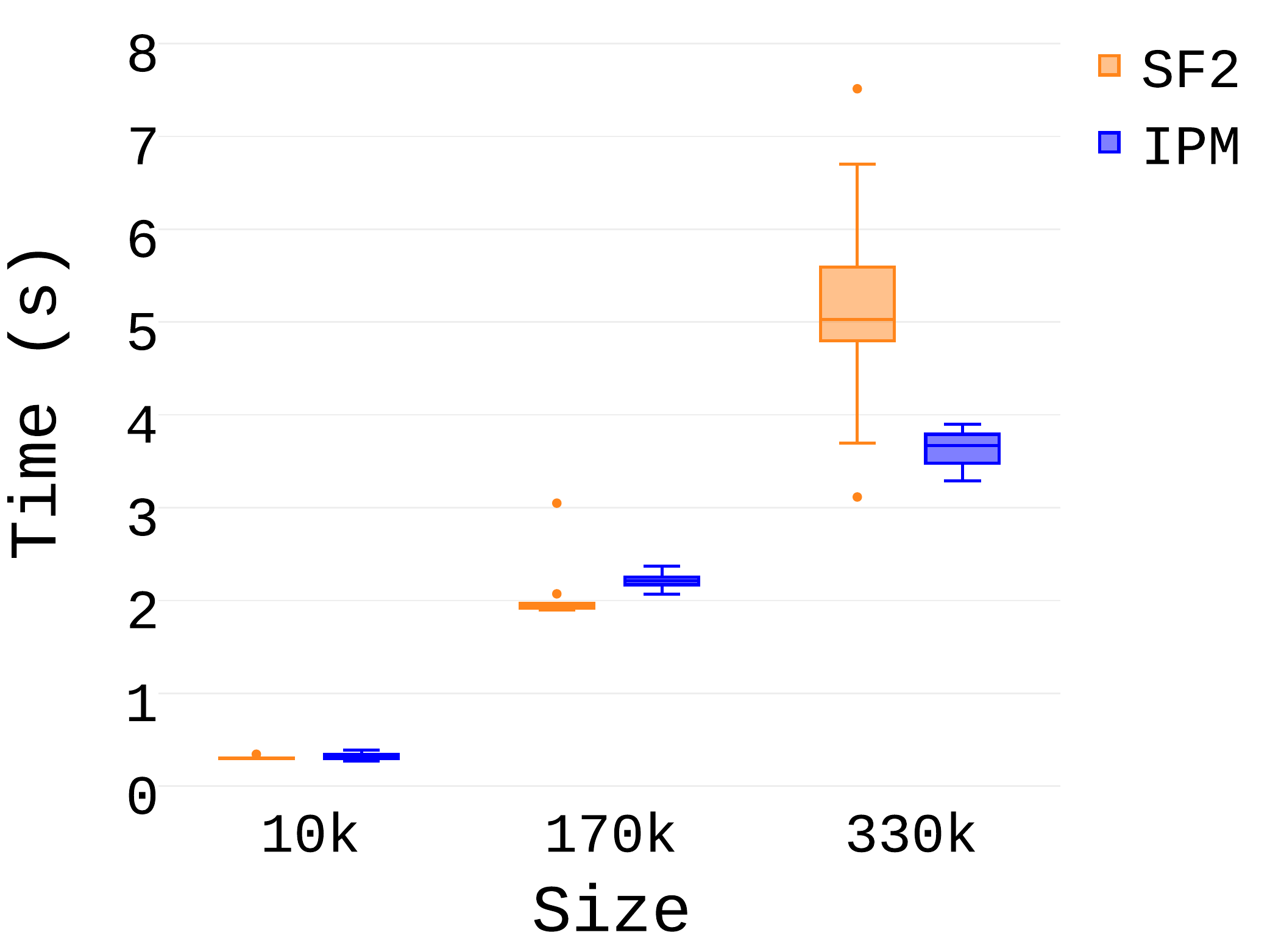}
\includegraphics[width = 2.2in]{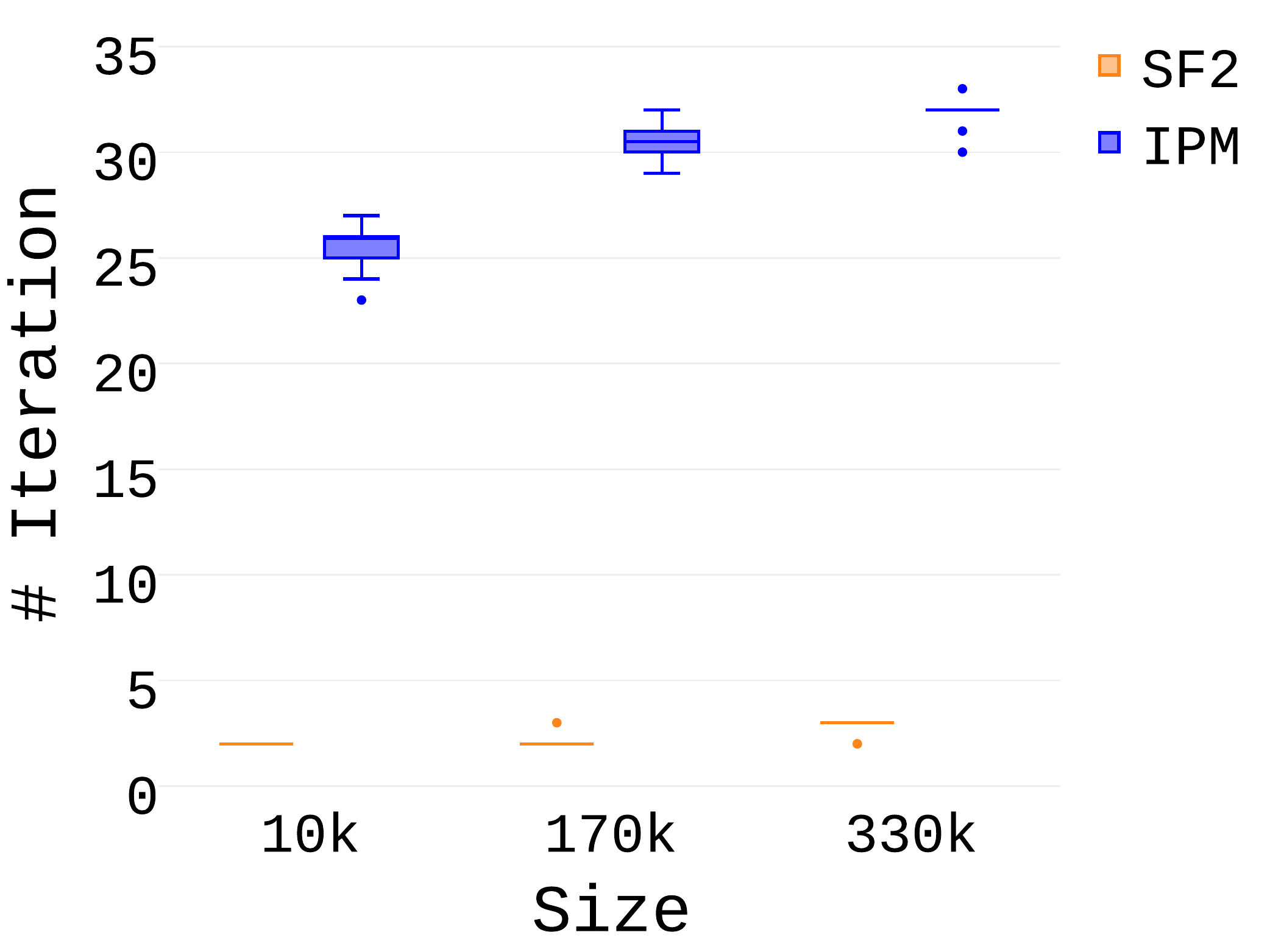}
}
\subfigure[$D = D^{(2,n)}$, $g=g_1$]{
\includegraphics[width = 2.2in]{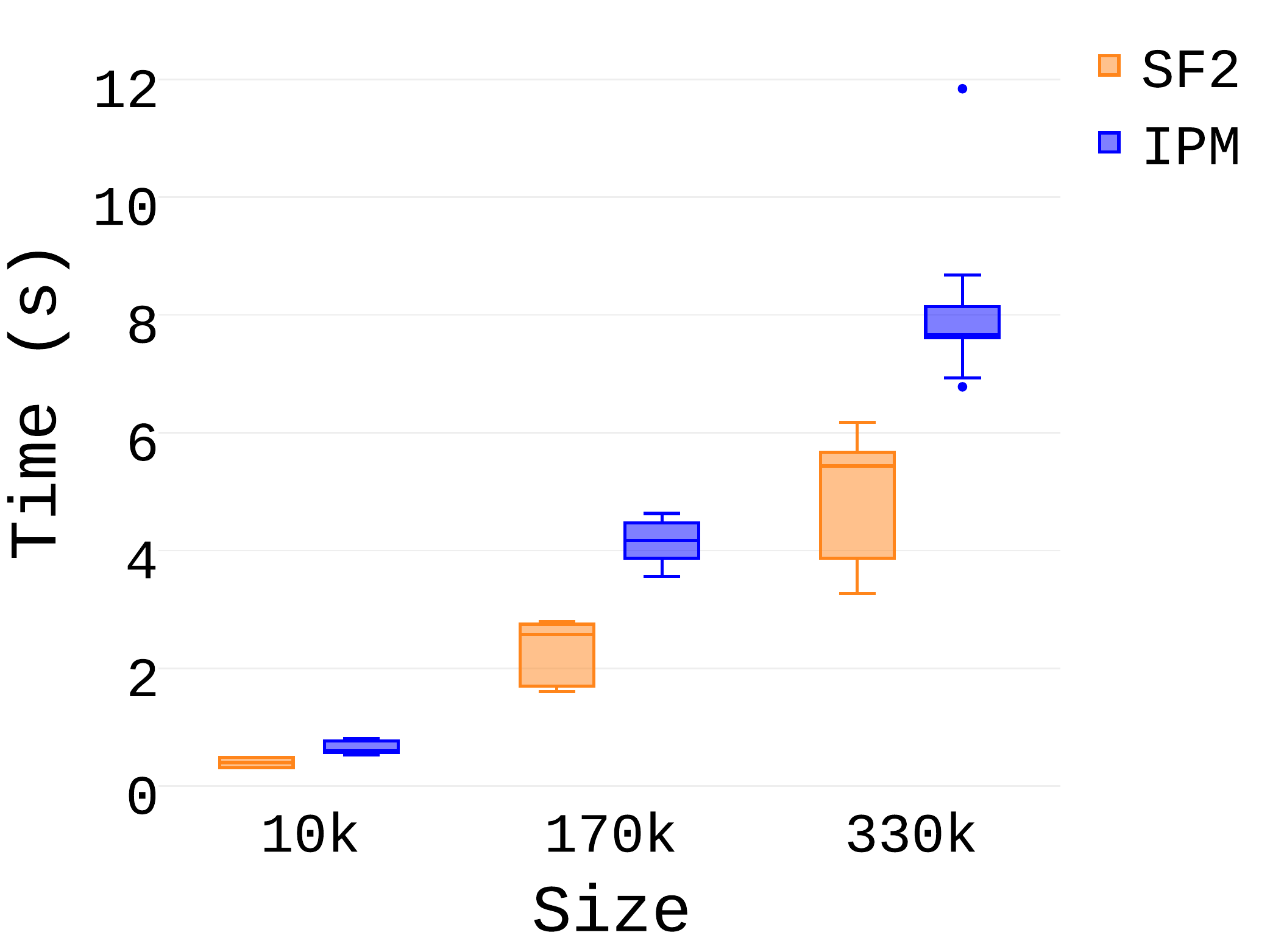}
\includegraphics[width = 2.2in]{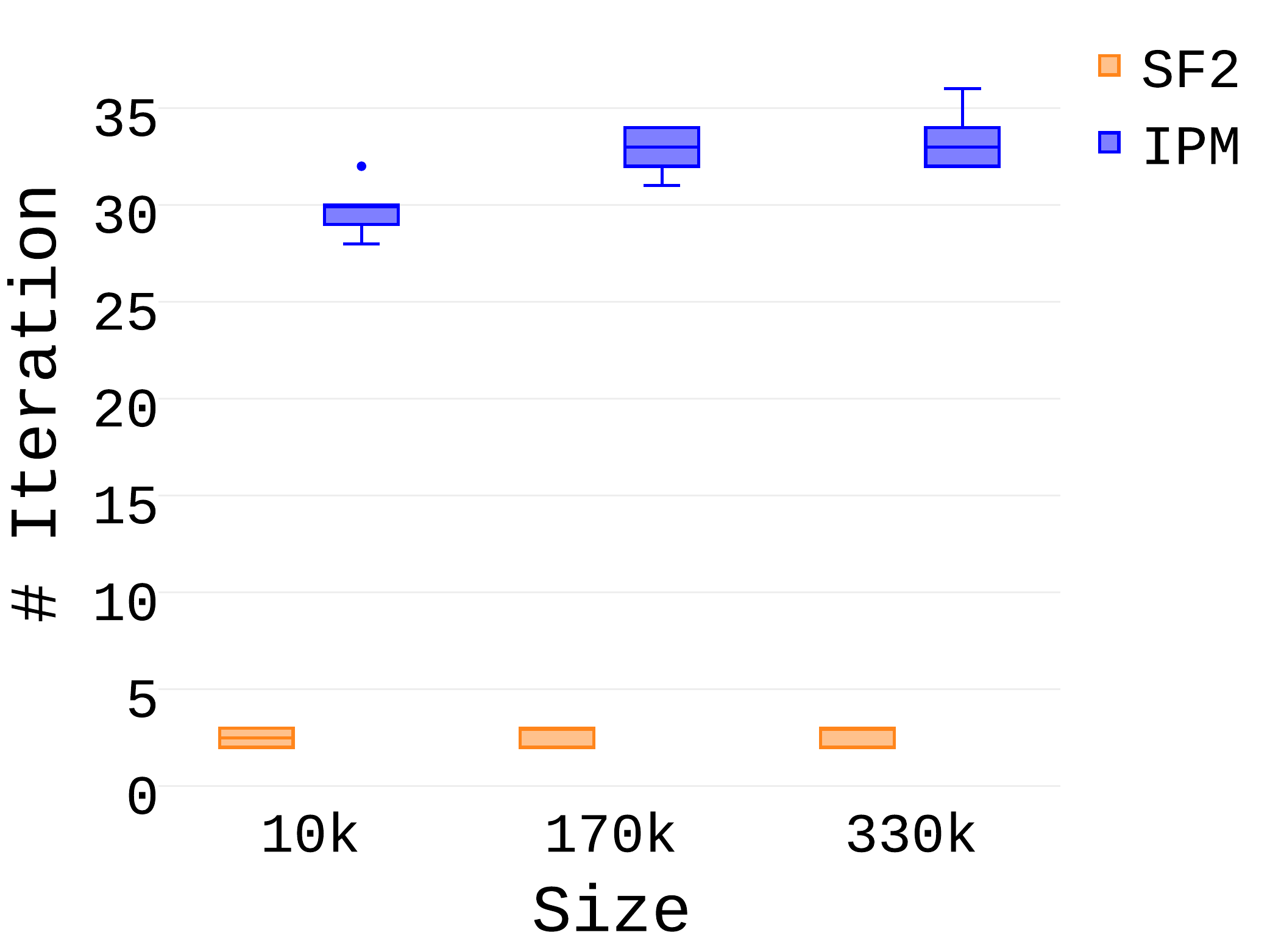}
}
\caption{PDAS (with warm-start) vs IPM.}
  \label{fig:tf.warm}
\end{figure}

Comparing the performance of \verb|SF2| between Figures~\ref{fig:tf1.cold}, \ref{fig:tf2.cold}, and \ref{fig:tf.warm} shows that warm-starting \verb|SF2| can dramatically reduce the cost of solving an instance of \eqref{prob.general}.  With a cold-start, \verb|SF2| may require hundreds of iterations, while with warm-starting it requires dramatically fewer iterations.  In contrast, \verb|IPM| does not benefit much from a good starting point.

\section{Concluding Remarks}\label{sec:conclusion}

We propose innovative PDAS algorithms for Isotonic Regression (IR) and Trend Filtering (TF).  For IR, our PDAS method enjoys the same theoretical properties as the well-known PAV method, but also has the ability to be warm-started, can exploit parallelism, and outperforms PAV in our experiments.  Our proposed safeguarding strategy for a PDAS method for TF also exhibits reliable and efficient behavior.  Overall, our proposed methods show that PDAS frameworks are powerful when solving a broad class of regularization problems.

\bibliographystyle{unsrt}
\bibliography{references}

\end{document}